\documentclass[12pt]{amsart}
\usepackage{amssymb,amsfonts,amsmath,amsthm,cite,verbatim}
\usepackage{xcolor}
\usepackage{tikz}
\usepackage{graphicx}
\usepackage[left=2.5cm, right=2.5cm, top=3.5cm, bottom=3.5cm]{geometry}

\usepackage{hyperref}

\usepackage[normalem]{ulem}
\newtheorem{theorem}{Theorem}[section]
\newtheorem{thm}[theorem]{Theorem}
\newtheorem{prop}[theorem]{Proposition}
\newtheorem{lem}[theorem]{Lemma}

\newtheorem{cor}[theorem]{Corollary}

\makeatletter \@addtoreset{equation}{section}

\newcommand{\qbinom}[2]{\genfrac{[}{]}{0pt}{}{#1}{#2}}

\newcommand{\Mid}{\:|\:}  
\DeclareMathOperator*{\CT}{CT}

\newcommand{\CC}{\mathbb{C}}

\begin{document}

\title[$q$-AFLT identity]{The AFLT $q$-Morris constant term identity}

\author{Yue Zhou}

\address{School of Mathematics and Statistics, HNP-LAMA, Central South University,
Changsha 410075, P.R. China}

\email{zhouyue@csu.edu.cn}


\date{June 30, 2022}

\begin{abstract}
It is well-known that the Selberg integral is equivalent to the Morris constant term identity.
More generally, Selberg type integrals can be turned into constant term identities for Laurent polynomials.
In this paper, by extending the Gessel--Xin method of the Laurent series proof of constant term identities, we obtain an AFLT type $q$-Morris constant term identity.
That is a $q$-Morris type constant term identity for a product of two Macdonald polynomials.

\end{abstract}

\maketitle

\section{Introduction}\label{s-intr}

Define the $n$-dimensional Jackson-integral or $q$-integral over $[0,1]^n$:
\begin{equation*}
\int_{[0,1]^n}f(x_1,\dots,x_n)\mathrm{d}_qx_1\cdots \mathrm{d}_qx_n
:=(1-q)^n\sum_{v_1,\dots,v_n\geq 0}f(q^{v_1},\dots,q^{v_n})q^{v_1+\cdots+v_n},
\end{equation*}
where $0<q<1$ and $f:\mathbb{R}^n\rightarrow \mathbb{C}$ is a function such that the sum on the right is absolutely convergent.
In 1980, Askey \cite{Askey} conjectured the next $q$-analogue of the Selberg integral \cite{Selberg,FW}. For $\alpha,\beta\in \mathbb{C}\setminus \{0,-1,-2,\dots,\}$ such that $\mathrm{Re}(\alpha)>0$
and $\gamma$ a positive integer,
\begin{multline}\label{e-qSelberg}
\int_{[0,1]^n} \prod_{i=1}^n x_i^{\alpha-1}(qx_i)_{\beta-1}
\prod_{1\leq i<j\leq n}x_i^{2\gamma}(q^{1-\gamma}x_j/x_i)_{2\gamma}
\mathrm{d}_qx_1\cdots \mathrm{d}_qx_n\\
=q^{\gamma\alpha \binom{n}{2}+2\gamma^2\binom{n}{3}}
\prod_{i=0}^{n-1}\frac{\Gamma_q(\alpha+i\gamma)\Gamma_q(\beta+i\gamma)\Gamma_q\big(1+(i+1)\gamma\big)}
{\Gamma_q\big(\alpha+\beta+(n-1+i)\gamma\big)\Gamma_q(1+\gamma)},
\end{multline}
where for $z\in \mathbb{C}$
\[
(a)_{z}=(a;q)_{z}:=\frac{(a)_{\infty}}{(aq^z)_{\infty}}\ \
\]
is the $q$-shifted factorial
and $\Gamma_q(z)=(q)_{z-1}/(1-q)^{z-1}$ is the $q$-gamma function.
Here $(a)_{\infty}=(a;q)_{\infty}:=(1-a)(1-aq)\cdots$.
Askey's conjecture was proved independently by Habsieger \cite{Habsieger1988} and Kadell \cite{Kadell1988}.
Hence, we refer the integral \eqref{e-qSelberg}
as the Askey--Habsieger--Kadell $q$-Selberg integral.
Expressing the integral as a constant term identity, Habsieger and Kadell
obtained the following equivalent result \cite{XZ}.
For nonnegative integers $a,b,c$,
\begin{equation}\label{q-Morris}
\CT_x \prod_{i=1}^{n}\Big(\frac{x_{0}}{x_{i}}\Big)_a
\Big(\frac{qx_{i}}{x_{0}}\Big)_b\prod_{1\leq i<j\leq n}
\Big(\frac{x_{i}}{x_{j}}\Big)_c\Big(\frac{x_{j}}{x_{i}}q\Big)_c
=\prod_{i=0}^{n-1}\frac{(q)_{a+b+ic}(q)_{(i+1)c}}{(q)_{a+ic}(q)_{b+ic}(q)_{c}},
\end{equation}
where $x:=(x_1,\dots,x_n)$ and $\CT\limits_xf(x)$ means to take the constant term of the Laurent polynomial (series) $f(x)$.
The $q=1$ case of \eqref{q-Morris} was proved by Morris \cite{Morris1982} in his Ph.D thesis.
He also showed that the Morris identity is equivalent to the Selberg integral and conjectured
the above $q$-analogue identity. Since \eqref{q-Morris} was conjectured by Morris and proved
by Habsieger and Kadell, this identity is usually referred as the Habsieger--Kadell $q$-Morris identity.
Note that the Laurent polynomial in the left-hand side of \eqref{q-Morris} is homogeneous in $x_0,x_1,\dots,x_n$, we can take $x_0=1$ without affecting the constant term.
Hence, the operator $\CT\limits_x$ in \eqref{q-Morris} is equivalent to $\CT\limits_{x_0,\dots,x_n}$.
The same trick also applies to $A_n(a,b,c,\lambda,\mu)$ below.

Recently, Albion, Rains and Warnaar obtained an elliptic ALFT-Selberg integeral \cite{ARW}.
As a corollary, they got
a $q$-Selberg type integral for a product of two Macdonald polynomials \cite[Corollary 1.5]{ARW}, see \eqref{e-ARW} below.
Let $\lambda$ and $\mu$ be two partitions such that
the length of $\mu$ is $l$.
For $\alpha,\beta,q,t\in \mathbb{C}$ such that $|\beta|,|q|,|t|<1$,
\begin{align}\label{e-ARW}
&\frac{1}{n!(2\pi i)^n}\int_{\mathbb{T}^n}P_{\lambda}(x;q,t)P_{\mu}\Big(\Big[x+\frac{t-\beta}{1-t}\Big];q,t\Big)\\
&\qquad \qquad \times\prod_{i=1}^n\frac{(\alpha/x_i,qx_i/\alpha)_{\infty}}{(\beta/x_i,x_i)_{\infty}}
\prod_{1\leq i<j\leq n}\frac{(x_i/x_j,x_j/x_i)_{\infty}}{(tx_i/x_j,tx_j/x_i)_{\infty}}
\frac{\mathrm{d}x_1}{x_1}\cdots \frac{\mathrm{d}x_n}{x_n}\nonumber \\
& = \beta^{|\lambda|}t^{|\mu|}P_{\lambda}\Big(\Big[\frac{1-t^n}{1-t}\Big];q,t\Big)
P_{\mu}\Big(\Big[\frac{1-\beta t^{n-1}}{1-t}\Big];q,t\Big) \nonumber \\
&\qquad \qquad \times \prod_{i=1}^n\frac{(t,\alpha t^{n-m-i}q^{\lambda_i},\alpha t^{1-i}/\beta,qt^{i-1}\beta/\alpha)_{\infty}}
{(q,t^{i},\beta t^{i-1},\alpha t^{1-i}q^{\lambda_i}/\beta)_{\infty}}
\prod_{i=1}^n\prod_{j=1}^m\frac{(\alpha t^{n-i-j+1}q^{\lambda_i+\mu_j})}{(\alpha t^{n-i-j}q^{\lambda_i+\mu_j})_{\infty}},
\nonumber
\end{align}
where $\mathbb{T}$ is a positively oriented unit circle,
$(a_1,\dots,a_k)_{\infty}=\prod_{i=1}^k(a_i)_{\infty}$,
$m$ is an arbitrary integer such that $m\geq l$,
$P_{\lambda}(x;q,t)$ is the Macdonald polynomial,
and $f[t+z]$ is plethystic notation for the symmetric function $f$.
Using the Cauchy residue theorem and take $t=q^c, \alpha=q^{b+1},\beta=q^{a+b+1},x_i\mapsto x_iq^{b+1}$ for all the $i$, it is not hard to transform the above integral into
a constant term identity, i.e., Theorem~\ref{thm-AFLT} below.

For nonnegative integers $a,b,c$, and partitions $\lambda$ and $\mu$ such that
the length of $\mu$ is $l$, denote
\begin{align}\label{Defi-A}
A_n(a,b,c,\lambda,\mu)&:=\CT_x x_0^{-|\lambda|-|\mu|}P_{\lambda}(x;q,q^c)
P_{\mu}\Big(\Big[\frac{q^{c-b-1}-q^a}{1-q^c}x_0+\sum_{i=1}^nx_i\Big];q,q^c\Big)\\
&\qquad \times
\prod_{i=1}^{n}\Big(\frac{x_{0}}{x_{i}}\Big)_a
\Big(\frac{qx_{i}}{x_{0}}\Big)_b\prod_{1\leq i<j\leq n}
\Big(\frac{x_{i}}{x_{j}}\Big)_c\Big(\frac{x_{j}}{x_{i}}q\Big)_c.\nonumber
\end{align}
Our main result in this paper is a direct constant term proof of the next AFLT type $q$-Morris constant term identity.
\begin{thm}\label{thm-AFLT}
Let $A_n(a,b,c,\lambda,\mu)$ be defined in \eqref{Defi-A}. Then
\begin{multline}\label{AFLT}
A_n(a,b,c,\lambda,\mu)=(-1)^{|\lambda|}q^{\sum_{i=1}^n\binom{\lambda_i}{2}-cn(\lambda)}
P_{\lambda}\Big(\Big[\frac{1-q^{nc}}{1-q^c}\Big];q,q^c\Big)
P_{\mu}\Big(\Big[\frac{q^{c-b-1}-q^{a+nc}}{1-q^c}\Big];q,q^c\Big) \\
\times
\prod_{i=1}^n\prod_{j=1}^l(q^{b+(n-i-j)c+\lambda_i+\mu_{j+1}+1})_{\mu_j-\mu_{j+1}}
\prod_{i=1}^n \frac{(q^{a+(i-1)c-\lambda_i+1})_{b+\lambda_i}(q)_{ic}}{(q)_{b+(n-i)c+\lambda_i+\mu_1}(q)_c},
\end{multline}
where $|\lambda|:=\sum_{i=1}^n\lambda_i$ is the size of $\lambda$,
and $n(\lambda):=\sum_{i=1}^n(i-1)\lambda_i$.
\end{thm}
Note the right-hand side of \eqref{AFLT} can be written as a compact product of $q$-factorials
by \eqref{e-special} below.
Many specializations of \eqref{AFLT} have rich history.
When $\lambda=\mu=0$, \eqref{AFLT} is the Habsieger--Kadell $q$-Morris identity \eqref{q-Morris} above.
When $t=q^{\gamma}$ and take $q \rightarrow 1$, the Macdonald polynomial $P_{\lambda}(x;q,t)$ reduces to the Jack polynomial $P_{\lambda}^{(1/\gamma)}(x)$.
Kadell \cite{Kadell97} generalized the Morris identity by adding a Jack polynomial, i.e., the $q=1$ and $\mu=0$ case of \eqref{AFLT}.
Then, he raised an open problem to find the $q$-analogue. The question
was solved by Macdonald by developing the theory of Macdonald polynomials \cite{MacSMC,Mac95}.
Macdonald gave an explicit $q$-analogue integral of Kadell's open problem in his book \cite[Page 374]{Mac95}.
By the standard transformation between the $q$-Selberg type integrals and the $q$-Morris type constant term identities \cite{Kadell1988,BF}, it is not hard to obtain a constant term identity equivalent to Macdonald's result. That is the $\mu=0$ case of \eqref{AFLT}.
In 1993, Kadell \cite{Kadell93} gave a two Jack symmetric function generalization of the Morris identity. Kadell's result in 1993 corresponds to the $q=1$ and $c=a+b+1$ case of \eqref{AFLT}. If $c=a+b+1$, the constant term identity \eqref{AFLT} reduces to \cite[Theorem 1.7]{Warnaar05}. It should be noted that the $q=1$ case of \eqref{AFLT} is equivalent to
the next ALFT type Selberg integral \cite{AFLT11}.
\begin{align}\label{e-ALFT}
&\int_{[0,1]^n}P_{\lambda}^{(1/\gamma)}(x)P^{(1/\gamma)}_{\mu}[x+\beta/\gamma-1]
\prod_{i=1}^n x_i^{\alpha-1}(1-x_i)^{\beta-1}\prod_{1\leq i<j\leq n}|x_i-x_j|^{2\gamma}
\mathrm{d}x_1\cdots \mathrm{d}x_n\\
&\quad
=P_{\lambda}^{(1/\gamma)}[n]P_{\mu}^{(1/\gamma)}[n+\beta/\gamma-1]
\prod_{i=1}^n\frac{\Gamma(\beta+(i-1)\gamma)\Gamma(\alpha+(n-i)\gamma+\lambda_i)\Gamma(1+i\gamma)}
{\Gamma(\alpha+\beta+(2n-l-i-1)\gamma+\lambda_i)\Gamma(1+\gamma)}\nonumber\\
&\qquad\times
\prod_{i=1}^n\prod_{j=1}^l\frac{\Gamma(\alpha+\beta+(2n-i-j-1)\gamma+\lambda_i+\mu_j)}
{\Gamma(\alpha+\beta+(2n-i-j)\gamma+\lambda_i+\mu_j)},\nonumber
\end{align}
where
\[
\mathrm{Re}(\alpha)>-\lambda_n, \quad
\mathrm{Re}(\beta)>0,\quad
\mathrm{Re}(\gamma)>-\min_{1\leq i\leq n-1}\Big\{\frac{1}{n},\frac{\mathrm{Re}(\alpha)+\lambda_i}{n-i},
\frac{\mathrm{Re}(\beta)}{n-1}\Big\}.
\]
Since Hua \cite{Hua} discovered the $\gamma=1$ and $\beta=\gamma$ case of the above integral,
and Kadell \cite{Kadell93} obtained the $\beta=\gamma$ case for general $\gamma$,
the $\beta=\gamma$ case of \eqref{e-ALFT} is called the Kadell--Hua integral.

The idea to prove Theorem~\ref{thm-AFLT}, is based on
the well-known fact that to prove the equality of two polynomials of degree at most $d$, it is sufficient to prove that they agree at $d+1$ distinct points.
We briefly outline the key steps.\vskip 0.2cm

\begin{enumerate}
\item \textbf{Polynomiality}.

It is routine to show that the constant term
$A_{n}(a,b,c,\lambda,\mu)$ (we refer the constant term as $A$ for short in this section) is a polynomial in $q^{a}$, assuming that all parameters but $a$ are fixed.
Then, we can extend the definition of $A$ for negative $a$, especially negative integers.
\vskip 0.2cm

\item \textbf{Rationality}
By a rationality result, see Corollary~\ref{cor-rationality} below, it suffices to prove \eqref{AFLT} for $c\geq l$. Here $l$ can be any nonnegative integer independent of $c$.
\vskip 0.2cm

\item \textbf{Determination of roots}

Let
\begin{align}\label{Roots-A}
A_1&=\{-ic-1,-ic-2,\dots,-ic-b\Mid i=0,1,\dots,n-1\},\\
A_2&=\{-(i-1)c+\lambda_{i}-1,-(i-1)c+\lambda_{i}-2,\dots,-(i-1)c\Mid i=1,\dots,\ell(\lambda)\},\nonumber\\
A_3&=\{-(n-j)c-b-1,-(n-j)c-b-2,\dots,-(n-j)c-b-\mu_{j}\Mid j=1,\dots,\ell(\mu)\},\nonumber
\end{align}
where $\ell(\lambda)$ is the length of the partition $\lambda$.
Suppose $c>b+\lambda_1+\mu_1$ in \eqref{Roots-A},
then all the elements of $A_1\cup A_2\cup A_3$ are distinct.

For $A$ viewed as a polynomial in $q^a$, we will determine all its roots under the assumption that $c$ is a sufficiently large integer (so that all the roots of $A$ are distinct).
To be precise, $A$ vanishes only when
$a$ equals one of the values in $A_1\cup A_2\cup A_3$.
\vskip 0.2cm

\item \textbf{Value at an additional point.}

We characterize the expressions for $A$ at $a=-b-1$ if $\ell(\mu)<n$
and at $a=-(n-\ell(\mu)-1)c-b-1$ if $\ell(\mu)\geq n$ respectively.

\end{enumerate}
The steps (1) and (2) are routine. A similar argument appeared in \cite{GX} and \cite{XZ}. The Step (3) is lengthy. To carry out the details, we need to mix three tools:
the iterated Laurent series, plethystic notation and substitution, and Cai's splitting idea
for Laurent polynomials. We have combined the first two tools in \cite{Zhou} to prove and extend Kadell's orthogonality conjecture. In Step (4), we reduce the expression for $A$ at the additional point to a similar type. We can uniquely determine a closed-form expression for $A$ after completing the above four steps.

This paper is organised as follows. In the next section, we introduce the basic notation used in this paper. In Section~\ref{sec-ple}, we give a brief introduction to a commonly used tool in the ring of symmetric functions --- the plethystic notation.
In Section~\ref{sec-iterate}, we present the essential material in the field of iterated Laurent series. In Section~\ref{sec-vanish}, we find that a family of constant terms vanish.
In Section~\ref{sec-Mac}, we obtain necessary results of Macdonald polynomials.
In Section~\ref{sec-poly}, we show the polynomiality and rationality of $A_n(a,b,c,\lambda,\mu)$. In Section~\ref{sec-prepare}, we prepare the preliminary for the
determination of the roots of $A_n(a,b,c,\lambda,\mu)$.
In the last section, we prove Theorem~\ref{thm-AFLT}.

\section{Basic notation}\label{sec-2}

In this section we introduce some basic notation used throughout this paper.

A partition $\lambda=({\lambda_1,\lambda_2,\dots)}$ is a sequence of decreasing nonnegative integers such that
only finitely-many $\lambda_i$ are positive.
The length of a partition $\lambda$, denoted
$\ell(\lambda)$ is the number of nonzero $\lambda_i$ (such nonzero $\lambda_i$ are called parts of $\lambda$).
The tails of zeros of a partition is usually not displayed.
If there are exactly $m_i$ of the parts of $\lambda$ are equal to $i$,
we can also denote a partition by
\[
\lambda=(1^{m_1}2^{m_2}\cdots r^{m_r}\cdots).
\]
The Young diagram of a partition is a collection of left-aligned rows of squares such that the $i$th row contains $\lambda_i$ squares. For example, the partition $(7,4,3,1)$ corresponds to
\begin{center}
\includegraphics[height=0.15\textwidth]{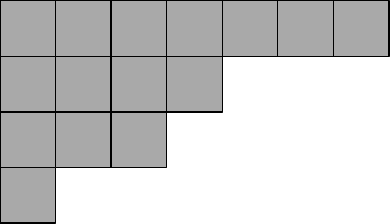}
\end{center}
If $\lambda,\mu$ are partitions, we shall write $\mu \subset \lambda$ if $\mu_i\leq \lambda_i$
for all $i\geq 1$. We can construct a skew diagram $\lambda/\mu$ whenever $\mu \subset \lambda$
by removing the squares of $\mu$ from those of $\lambda$.
For example, if $\lambda=(7,4,3,1)$ and $\mu=(4,3,1)$, then
$\mu \subset \lambda$ and the skew diagram $\lambda/\mu$ is the following:
\begin{center}
\includegraphics[height=0.15\textwidth]{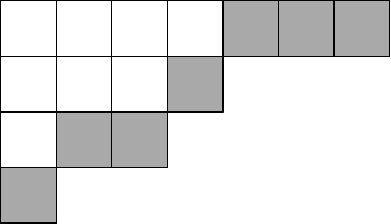}
\end{center}
A skew diagram $\theta$ is said to be a horizontal $r$-strip if $|\theta|=r$ and it contains at most one square in every column.
For example, the above diagram for $(7,4,3,1)/(4,3,1)$ is a horizontal 7-strip.
The size of the partition $\lambda$ is $|\lambda|=\lambda_1+\lambda_2+\cdots$.
If $\lambda,\mu$ are partitions of the same size then we write $\mu\leq \lambda$ if
$\mu_1+\cdots+\mu_i\leq \lambda_1+\cdots+\lambda_i$ for all $i\geq 1$.
This partial order on the set of partitions of the same size is called the dominance order.
As usual, we write $\mu<\lambda$ if $\mu\leq \lambda$ but $\mu\neq \lambda$.

The infinite $q$-shifted factorial is defined as
\[
(z)_{\infty}=(z;q)_{\infty}:=
\prod_{i=0}^{\infty}(1-zq^i)
\]
where, typically, we suppress the base $q$.
Then, for $k$ an integer,
\[
(z)_k=(z;q)_k:=\frac{(z;q)_{\infty}}
{(zq^k;q)_{\infty}}.
\]
Note that
\[
(z)_k=
\begin{cases}
(1-z)(1-zq)\cdots (1-zq^{k-1}) & \text{if $k>0,$}
\\
1 & \text{if $k=0,$}\\
\displaystyle
\frac{1}{(1-zq^k)(1-zq^{k+1})\cdots(1-zq^{-1})}=\frac{1}{(zq^k)_{-k}}
& \text{if $k<0$.}
\end{cases}
\]
Using the above we define the
$q$-binomial coefficient as
\[
\qbinom{n}{k}=\frac{(q^{n-k+1})_k}{(q)_k}
\]
for $n$ an arbitrary integer and $k$ a nonnegative integer.

By convention, if $k<j$ then we set $\sum_{i=j}^k n_i:=0$ and $\prod_{i=j}^k n_i:=1$.

\section{Plethystic notation}\label{sec-ple}

In this section, we briefly introduce plethystic notation and substitutions.
For more details, see \cite{haglund,Lascoux,Mellit,RW}.

Denote by $\Lambda_{\mathbb{F}}$ the ring of symmetric functions in countably many variables with coefficients in a field $\mathbb{F}$.
For an alphabet $X=(x_1,x_2,\dots)$,
we additively write $X:=x_1+x_2+\cdots$, and
use plethystic brackets to indicate this additive notation:
\[
f(X)=f(x_1,x_2,\dots)=f[x_1+x_2+\cdots]=f[X], \quad
\text{for $f\in \Lambda_{\mathbb{F}}$.}
\]

For $r$ a positive integer, define the power sum symmetric function $p_r$ in the alphabet $X$ as
\[
p_r=\sum_{i\geq 1}x_i^r,
\]
and $p_0=1$.
For a partition $\lambda=(\lambda_1,\lambda_2,\dots)$, let
\[
p_{\lambda}=p_{\lambda_1}p_{\lambda_2}\cdots.
\]
The $p_{\lambda}$ form a basis of $\Lambda_{\mathbb{Q}}$ and the $p_r$ are algebraically independent over
$\mathbb{Q}$ \cite{Mac95}. That is,
\[
\Lambda_{\mathbb{Q}}=\mathbb{Q}[p_1,p_2,\dots].
\]
For two alphabets $X$ and $Y$, define
\begin{subequations}\label{asm}
\begin{align}
\label{asm1}
p_r[X+Y]&=p_r[X]+p_r[Y], \\
p_r[X-Y]&=p_r[X]-p_r[Y], \\
p_r[XY]&=p_r[X]p_r[Y],\label{asm3}
\end{align}
and a particular division
\begin{equation}\label{division}
p_r\Big[\frac{X}{1-t}\Big]=\frac{p_r[X]}{1-t^r}.
\end{equation}
\end{subequations}
Note that the alphabet $1/(1-t)$ may be interpreted as
the infinite alphabet $1+t+t^2+\cdots$.
If $f\in \Lambda_{\mathbb{Q}}$ and $f=\sum_{\lambda}c_{\lambda}p_{\lambda}$,
we can extend the above plethystic notation to any symmetric function $f$
by
\[
f[X]=\sum_{\lambda}c_{\lambda}\prod_{i\geq 1}p_{\lambda_i}[X].
\]
Here the $X$ can be any alphabet satisfies the rules in \eqref{asm}.
For example, we can take $X\mapsto Y+Z$ and $X\mapsto YZ$.

Occasionally we need to use an ordinary minus sign in plethystic notation.
To distinguish this from a plethystic minus sign, we denote by $\epsilon$ the alphabet consisting of the single letter $-1$, so that for $f\in \Lambda_{\mathbb{F}}$
\[
f(-x)=f(-x_1,-x_2,\dots)=f[\epsilon x_1+\epsilon x_2+\cdots]=f[\epsilon X].
\]
Hence
\[
p_r[\epsilon X]=(-1)^rp_r[X], \quad p_r[-\epsilon X]=(-1)^{r-1}p_r[X].
\]
If $f$ is homogeneous symmetric function of degree $k$ then
\begin{equation}\label{e-homo-sym}
f[aX]=a^kf[X]
\end{equation}
for a single-letter alphabet $a$.
In particular,
\begin{equation}\label{e-Mac3}
f[\epsilon X]=(-1)^kf[X].
\end{equation}
From \eqref{asm1}, we also have
\[
p_r[2X]:=p_r[X+X]=2p_r[X].
\]
We extend the above to any $k\in \mathbb{F}$ via
\begin{equation}
p_r[kX]=kp_r[X].
\end{equation}
Note that this may lead to some notational ambiguities.
We will indicate that a symbol such as $a$ or $k$
represents a letter or a binomial element\footnote{In \cite[p. 32]{Lascoux} Lascoux referred to $k\in \mathbb{F}$ as a binomial element.} whenever not clear.

\section{Constant term evaluations using iterated Laurent series}\label{sec-iterate}

In this section we present a basic lemma for extracting
constant terms from rational functions in the field of iterated Laurent series.
In this paper, we denote $K=\CC(q)$ and work in the field of iterated
Laurent series $K\langle\!\langle x_n, x_{n-1},\dots,x_0\rangle\!\rangle
=K(\!(x_n)\!)(\!(x_{n-1})\!)\cdots (\!(x_0)\!)$.
That is, elements of this field are regarded first as Laurent series in $x_0$, then as
Laurent series in $x_1$, and so on.
For a more detailed account of the properties of this field,
see \cite{xinresidue} and~\cite{xiniterate}.
A crucial fact is that the field $K(x_0,\dots,x_n)$ of
rational functions forms a subfield of
$K\langle\!\langle x_n, x_{n-1},\dots,x_0\rangle\!\rangle$, so that every rational function is identified as its unique Laurent series expansion in the field of iterated Laurent series.

The following series expansion of $1/(1-cx_i/x_j)$
for $c\in K\setminus \{0\}$ forms a key ingredient in our approach:
\[
\frac{1}{1-c x_i/x_j}=
\begin{cases} \displaystyle \sum_{l\geq 0} c^l (x_i/x_j)^l
& \text{if $i<j$}, \\[5mm]
\displaystyle -\sum_{l<0} c^l (x_i/x_j)^l
& \text{if $i>j$}.
\end{cases}
\]
Thus, the constant term of $1/(1-c x_i/x_j)$ with respect to $x_i$
is $1$ if $i<j$ and $0$ if $i>j$.
That is,
\begin{equation}
\label{e-ct}
\CT_{x_i} \frac{1}{1-c x_i/x_j} =
\begin{cases}
    1 & \text{if $i<j$}, \\
    0 & \text{if $i>j$}, \\
\end{cases}
\end{equation}
where, for $f\in K\langle\!\langle x_n, x_{n-1},\dots,x_0\rangle\!\rangle$,
we use the notation $\CT\limits_{x_i} f$ to denote taking the constant term
of $f$ with respect to $x_i$.
An important property of the constant term operators defined this
way is their commutativity:
\[
\CT_{x_i} \CT _{x_j} f = \CT_{x_j} \CT_{x_i} f.
\]
This means that we take constant terms with respect to a set of variables.

The following lemma is a basic tool for
extracting constant terms from rational functions. It first
appeared in~\cite{GX}. We restated this lemma in \cite{Zhou}.
\begin{lem}\label{lem-prop}
For a positive integer $m$, let $p(x_k)$ be a Laurent polynomial
in $x_k$ of degree at most $m-1$ with coefficients in
$K\langle\!\langle x_n,\dots,x_{k-1},x_{k+1},\dots,x_0\rangle\!\rangle$.
Let $0\leq i_1\leq\dots\leq i_m\leq n$ such that all $i_r\neq k$,
and define
\begin{equation}\label{e-defF}
f=p(x_k)/D,
\end{equation}
where $D$ is of the form
\begin{equation}\label{e-formD}
\prod_{r=1}^m (1-c_r x_k/ x_{i_r}),
\end{equation}
and $c_1,\dots,c_m\in K\setminus \{0\}$ such that $c_r\neq c_s$ if $x_{i_r}=x_{i_s}$.
Then
\begin{equation}\label{e-almostprop}
\CT_{x_k} f=\sum_{\substack{r=1 \\[1pt] i_r>k}}^m
\big(f\,(1-c_r x_k/x_{i_r})\big)\Big|_{x_k=c_r^{-1}x_{i_r}}.
\end{equation}
\end{lem}

At the end of this section, we obtain an equivalence between two kinds of constant terms.
\begin{prop}\label{prop-equiv}
Let $f(x)$ be a rational function invariant under any permutation of $x$. Then
\begin{equation}\label{e-equiv}
\CT_x f(x)\prod_{1\leq i<j\leq n}(x_i/x_j)_c(qx_j/x_i)_c
=\frac{1}{n!}\prod_{i=1}^{n-1}\frac{1-q^{(i+1)c}}{1-q^c}\times\CT_x f(x)\prod_{1\leq i\neq j\leq n} (x_i/x_j)_c.
\end{equation}
\end{prop}
\begin{proof}
For $w:=(w_1,\dots,w_n)\in \mathfrak{S}_n$ and any function $g(x)$, define
\[
w\circ g(x):=g(x_{w_1},\dots,x_{w_n}).
\]
Since $w$ acts on any Laurent polynomial (or Laurent series) does not change the constant term, we have
\begin{multline}\label{e-equiv1}
\CT_x f(x)\prod_{1\leq i<j\leq n}(x_i/x_j)_c(qx_j/x_i)_c \\
\quad =\frac{1}{n!}\sum_{w\in \mathfrak{S}_n}\CT_x w\circ \bigg(f(x)\prod_{1\leq i\neq j\leq n} (x_i/x_j)_c
\prod_{1\leq i<j\leq n}\frac{1-q^cx_j/x_i}{1-x_j/x_i}\bigg).
\end{multline}
Because $f(x)\prod_{1\leq i\neq j\leq n} (x_i/x_j)_c$ is invariant under the action of $w$, the right-hand side of \eqref{e-equiv1} becomes
\[
\frac{1}{n!}\CT_x f(x)\prod_{1\leq i\neq j\leq n} (x_i/x_j)_c \cdot \sum_{w\in \mathfrak{S}_n}w\circ
\bigg(\prod_{1\leq i<j\leq n}\frac{1-q^cx_j/x_i}{1-x_j/x_i}\bigg),
\]
which is the right-hand side of \eqref{e-equiv} by the next identity \cite[III, (1.4)]{Mac95}.
\begin{equation}\label{e-equiv2}
\sum_{w\in \mathfrak{S}_n}w\circ
\bigg(\prod_{1\leq i<j\leq n}\frac{1-q^cx_j/x_i}{1-x_j/x_i}\bigg)
=\prod_{i=1}^{n-1}\frac{1-q^{(i+1)c}}{1-q^c}.\qedhere
\end{equation}
\end{proof}

\section{A family of vanishing constant terms}\label{sec-vanish}

In this section, we find that a family of constant terms vanish using the splitting formula obtained by Cai \cite{cai}.

Let $w$ be a parameter such that
all terms of the form $dz_i/w$
satisfy $|dz_i/w|<1$, where $d\in \{1,q,q^2,\dots,q^{c-1}\}$. Hence,
\[
\frac{1}{1-dz_i/w}=\sum_{j\geq 0} (dz_i/w)^j.
\]
Then it is easy to see that
\begin{equation}\label{e-relation}
\CT_z z^v\prod_{1\leq i<j\leq n}\big(z_i/z_j\big)_c\big(qz_j/z_i\big)_c
=\CT_{z,w}\frac{z^v\prod_{1\leq i<j\leq n}\big(z_i/z_j\big)_c\big(qz_j/z_i\big)_c}{\prod_{i=1}^n(z_i/w)_c},
\end{equation}
where $v:=(v_1,v_2,\dots,v_n)\in \mathbb{Z}^n$, $z=(z_1,\dots,z_n)$, and
$z^v$ denotes the monomial $z_1^{v_1}\cdots z_n^{v_n}$.
Note that if $|v|:=v_1+\cdots+v_n\neq 0$ then
the constant term in the left-hand side of \eqref{e-relation} vanishes by the homogeneity.

To obtain the main result in this section, we need the equal parameter case of \cite[Theorem 3.3]{cai}.
\begin{prop}\label{prop-split}
We can write
\begin{equation}\label{e-Fsplit}
\frac{\prod_{1\leq i<j\leq n}\big(z_i/z_j\big)_c\big(qz_j/z_i\big)_c}{\prod_{i=1}^n(z_i/w)_c}=\sum_{i=1}^{n}\sum_{j=0}^{c-1}\frac{A_{ij}}{1-q^jz_i/w},
\end{equation}
where
\begin{multline}\label{A}
A_{ij}=\frac{q^{c\big((j+1)n-i-j\big)}}{(q^{-j})_j(q)_{c-j-1}}
\prod_{l=1}^{i-1}\big(q^{1-c}z_i/z_l\big)_j\big(q^{j+1}z_i/z_l\big)_{c-j}
\prod_{l=i+1}^{n}\big(q^{-c}z_i/z_l\big)_{j+1}\big(q^{j+1}z_i/z_l\big)_{c-j-1}
\\
\times
\prod_{\substack{1\leq u<v\leq n\\u,v\neq i}}\big(z_u/z_v\big)_c\big(qz_v/z_u\big)_c.
\end{multline}
\end{prop}
Note that all the $A_{ij}$ are polynomials in $z_i$.

For a nonnegative integer $r$ and an alphabet $X:=\{x_1,x_2,\dots\}$, define the $r$-th complete symmetric function $h_r(X)$
in terms of its generating function as
\begin{equation}\label{e-gfcomplete}
\sum_{r\geq 0} t^r h_r(X)=\prod_{i\geq 1}
\frac{1}{1-tx_i}.
\end{equation}
Using Proposition~\ref{prop-split}, we find that a family of constant terms vanish.
\begin{lem}\label{lem-positiveroots}
Let $v=(v_1,\dots,v_n)\in \mathbb{Z}^n$ and $\lambda$ be a partition such that $|v|=|\lambda|$. For
$\lambda_1>\max\{v_i\mid i=1,\dots,n\}$,
\begin{equation}
\CT_z\frac{1}{z^v}h_{\lambda}\Big[\sum_{i=1}^{n}\frac{1-q^c}{1-q}z_i\Big]
\prod_{1\leq i<j\leq n}\big(z_i/z_j\big)_c\big(qz_j/z_i\big)_c
=0.
\end{equation}
\end{lem}
Note that we can write $h_{\lambda}\big[\sum_{i=1}^{n}\frac{1-q^c}{1-q}z_i\big]=g_{\lambda}(z;q,q^c)$
for modified complete symmetric function by \eqref{modi-complete} below.
\begin{proof}
By the generating function for complete symmetric functions \eqref{e-gfcomplete},
\[
h_{r}\Big[\sum_{i=1}^{n}\frac{1-q^c}{1-q}z_i\Big]
=\CT_w \frac{w^r}{\prod_{i=1}^n(z_i/w)_c}
\]
for $r$ a nonnegative integer.
Thus
\begin{multline}\label{pr-1}
\CT_z\frac{1}{z^v}h_{\lambda}\Big[\sum_{i=1}^{n}\frac{1-q^c}{1-q}z_i\Big]
\prod_{1\leq i<j\leq n}\big(z_i/z_j\big)_c\big(qz_j/z_i\big)_c\\
=\CT_{z,w}\frac{w^{\lambda_1}h_{\lambda^{(1)}}\big[\sum_{i=1}^{n}\frac{1-q^c}{1-q}z_i\big]
\prod_{1\leq i<j\leq n}\big(z_i/z_j\big)_c\big(qz_j/z_i\big)_c}
{z^v\prod_{i=1}^n(z_i/w)_c},
\end{multline}
where $\lambda^{(1)}=(\lambda_2,\lambda_3,\dots)$.
Using \eqref{e-Fsplit}, we can write the right-hand side of \eqref{pr-1} as
\begin{equation}\label{pr-2}
\CT_{z,w}\sum_{i=1}^{n}\sum_{j=0}^{c-1}\frac{w^{\lambda_1}
h_{\lambda^{(1)}}\big[\sum_{t=1}^{n}\frac{1-q^{c}}{1-q}z_t\big]
A_{ij}}{z^v(1-q^jz_i/w)}.
\end{equation}
Taking the constant term with respect to $w$ in \eqref{pr-2} yields
\begin{equation}\label{pr-3}
\CT_{z}\sum_{i=1}^{n}\sum_{j=0}^{c-1}q^{j\lambda_1}z^{-v}z_i^{\lambda_1}
h_{\lambda^{(1)}}\Big[\sum_{t=1}^{n}\frac{1-q^{c}}{1-q}z_t\Big]A_{ij}.
\end{equation}
Since $\max\{v\}<\lambda_1$ and the $A_{ij}$ are polynomials in $z_i$, every term in the sums of \eqref{pr-3} is a polynomial in $z_i$ with no constant term.
Thus, the right-hand side of \eqref{pr-1} vanishes. Then the lemma follows.
\end{proof}

\section{Macdonald polynomials}\label{sec-Mac}

In this section, we give several essential results for Macdonald polynomials.

Given a sequence $\alpha=(\alpha_1,\alpha_2,\dots)$ of nonnegative integers such that $|\alpha|=\alpha_1+\alpha_2+\cdots$ is finite.
Let $X=\{x_1,x_2,\dots\}$ be an alphabet of countably many variables. Then the monomial symmetric function indexed by a partition $\lambda$ is defined as
\[
m_{\lambda}=m_{\lambda}(X):=\sum x_1^{\alpha_1}x_2^{\alpha_2}\cdots,
\]
where the sum is over all distinct permutations $\alpha$ of $\lambda$.
Denote $\mathbb{F}=\mathbb{Q}(q,t)$.
Let $\langle\cdot,\cdot\rangle:\Lambda_{\mathbb{F}}\times \Lambda_{\mathbb{F}} \rightarrow \mathbb{F}$
be the $q,t$-Hall scalar product on $\Lambda_{\mathbb{F}}$ given by \cite[page 306]{Mac95}
\[
\langle p_{\lambda}, p_{\mu}\rangle:=\delta_{\lambda,\mu}z_{\lambda}\prod_{i\geq 1}\frac{1-q^{\lambda_i}}{1-t^{\lambda_i}},
\]
where $z_{\lambda}:=\prod_{i\geq 1}i^{m_i}m_i!$ if $m_i$ is the number of parts $i$ in $\lambda$,
and the Kronecker symbol $\delta_{\lambda,\mu}=1$ if $\lambda=\mu$ and 0 otherwise.
The Macdonald polynomials
$P_{\lambda}=P_{\lambda}(q,t)=P_{\lambda}(X;q,t)$ are the unique symmetric functions \cite[VI, (4.7)]{Mac95} such that
\[
\langle P_{\lambda},P_{\mu}\rangle=0 \quad \text{if} \quad \lambda\neq \mu
\]
and
\begin{equation}\label{Mac-m}
P_{\lambda}=m_{\lambda}+\sum_{\mu<\lambda}u_{\lambda\mu}m_{\mu}, \quad u_{\lambda\mu}\in \mathbb{F}.
\end{equation}
Let
\[
b_{\lambda}=b_{\lambda}(q,t)=\langle P_{\lambda},P_{\lambda}\rangle^{-1},
\]
and
\begin{equation}\label{PQ}
Q_{\lambda}=b_{\lambda}P_{\lambda},
\end{equation}
so that
\[
\langle P_{\lambda}, Q_{\mu}\rangle=\delta_{\lambda \mu}
\]
for all partitions $\lambda$ and $\mu$. See \cite[VI, (4.11), (4.12)]{Mac95}.
For any three partitions $\lambda,\mu,\nu$ let
\[
f_{\mu \nu}^{\lambda}=f_{\mu \nu}^{\lambda}(q,t)=\langle Q_{\lambda},P_{\mu}P_{\nu}\rangle\in \mathbb{F}.
\]
The coefficient $f_{\mu \nu}^{\lambda}$ is called the $q,t$-Littlewood--Richardson coefficient.
Let $\lambda,\mu$ be partitions and define skew functions
$Q_{\lambda/\mu}\in \Lambda_{\mathbb{F}}$
by
\begin{equation}\label{skewQ1}
Q_{\lambda/\mu}=\sum_{\nu}f_{\mu \nu}^{\lambda}Q_{\nu}
\end{equation}
so that
\begin{equation}\label{skewQ}
\langle Q_{\lambda/\mu}, P_{\nu} \rangle=\langle Q_{\lambda}, P_{\mu}P_{\nu} \rangle.
\end{equation}
Likewise we define $P_{\lambda/\mu}$ by interchanging the $P$'s and $Q$'s in \eqref{skewQ}
\begin{equation}\label{skewP}
\langle P_{\lambda/\mu}, Q_{\nu} \rangle=\langle P_{\lambda}, Q_{\mu}Q_{\nu} \rangle.
\end{equation}
Since $Q_{\lambda}=b_{\lambda}P_{\lambda}$, it follows that
\begin{equation}\label{skewQP}
Q_{\lambda/\mu}=b_{\lambda}b_{\mu}^{-1}P_{\lambda/\mu}.
\end{equation}
See \cite[VI, (7.8)]{Mac95}.
For $u,v\in \mathbb{F}$ such that $v\neq \pm 1$, let $\omega_{u,v}$ denote the $\mathbb{F}$-algebra
endomorphism of $\Lambda_{\mathbb{F}}$ defined by
\[
\omega_{u,v}(p_{r})=(-1)^{r-1}\frac{1-u^r}{1-v^r}p_r
\]
for all $r\geq 1$, so that
\[
\omega_{u,v}(p_{\lambda})=(-1)^{|\lambda|-\ell(\lambda)}p_{\lambda}\prod_{i=1}^{\ell(\lambda)}\frac{1-u^{\lambda_i}}{1-v^{\lambda_i}}.
\]
Then we have the duality \cite[VI, (7.16)]{Mac95}
\begin{equation}\label{e-Mac1}
\omega_{q,t} P_{\lambda/\mu}(q,t)=Q_{\lambda'/\mu'}(t,q).
\end{equation}
For any alphabet $X$, the endomorphism $\omega_{q,t}$ can be interpreted plethystically.
That is
\begin{equation}\label{e-Mac2}
X\mapsto \epsilon X \frac{1-q}{t-1}.
\end{equation}
For two alphabets $X$ and $Y$,
\begin{equation}\label{e-Mac4}
P_{\lambda}[X+Y]
=\sum_{\mu\subset \lambda}P_{\lambda/\mu}[X]P_{\mu}[Y].
\end{equation}
See \cite[VI, (7.9')]{Mac95}.
If $n<\ell(\lambda)$ then
\begin{equation}\label{e-Mac5}
P_{\lambda}(x_1,\dots,x_n;q,t)=0.
\end{equation}
See \cite[VI, (4.10)]{Mac95}.
If $\lambda$ is a partition of length $n$, then
\begin{equation}\label{e-Mac7}
P_{\lambda}(x_1,\dots,x_n;q,t)=x_1\dots x_n P_{\mu}(x_1,\dots,x_n;q,t),
\end{equation}
where $\mu=(\lambda_1-1,\dots,\lambda_n-1)$.
See \cite[VI, (4.17)]{Mac95}.
For the skew functions
\begin{equation}\label{e-Mac6}
Q_{\lambda/\mu}(x_1,\dots,x_n;q,t)=0
\end{equation}
unless $0\leq \lambda'_i-\mu'_i\leq n$ for each $i\geq 1$.
See \cite[VII, (7.15)]{Mac95}.
A generalization of the $q$-factorial to partitions is given by
\[
(a;q,t)_{\lambda}:=\prod_{i\geq 1}(at^{1-i})_{\lambda_i}.
\]
The generalized hook polynomial is defined as
\[
c_{\lambda}(q,t)=\prod_{i=1}^{n}(t^{n-i+1})_{\lambda_i}\prod_{1\leq i<j\leq n}\frac{(t^{j-i})_{\lambda_i-\lambda_j}}{(t^{j-i+1})_{\lambda_i-\lambda_j}}
\]
for $n\geq \ell(\lambda)$.
Then
\begin{equation}\label{e-special}
P_{\lambda}\Big(\Big[\frac{1-a}{1-t}\Big];q,t\Big)
=\frac{t^{n(\lambda)}(a;q,t)_{\lambda}}{c_{\lambda}(q,t)}.
\end{equation}
As claimed in \cite{ARW}, this specialization is equivalent to \cite[p. 338]{Mac95}.

\begin{lem}\label{lem-Mac-1}
Let $\lambda$ and $\mu$ be partitions such that $\mu\subset \lambda$.
Then for $l$ a positive integer
\begin{equation}\label{lem-Mac-e2}
P_{\lambda/\mu}\Big(\Big[\frac{1-q}{t-1}\big(m_1+m_2+\cdots+m_l\big)\Big];q,t\Big)
=(-1)^{|\lambda|-|\mu|}Q_{\lambda'/\mu'}
\Big(\big[m_1+m_2+\cdots+m_l\big];t,q\Big),
\end{equation}
where the $m_i$ are single-letter alphabets.
\end{lem}
\begin{proof}
We prove by straightforward calculation.
\begin{align*}\label{lem-Mac-e1}
&Q_{\lambda'/\mu'}\Big(\big[m_1+m_2+\cdots+m_l\big];t,q\Big) \\
&\quad=\omega_{q,t}P_{\lambda/\mu}\Big(\big[m_1+m_2+\cdots+m_l\big];q,t\Big)
\quad &\text{by \eqref{e-Mac1}}\ \\
&\quad =P_{\lambda/\mu}\Big(\Big[\big(\epsilon \frac{1-q}{t-1}\big)\big(m_1+m_2+\cdots+m_l\big)\Big];q,t\Big) \quad &\text{by \eqref{e-Mac2}}\ \\
&\quad =(-1)^{|\lambda|-|\mu|}P_{\lambda/\mu}\Big(\Big[\frac{1-q}{t-1}\big(m_1+m_2+\cdots+m_l\big)\Big];q,t\Big) \quad & \text{by \eqref{e-Mac3}}.
\end{align*}
Then we obtain \eqref{lem-Mac-e2}.
\end{proof}

By Lemma~\ref{lem-Mac-1} we obtain the next vanishing property for Macdonald polynomials.
\begin{prop}\label{prop-Mac-vanish}
Let $\lambda$ be a partition and $\lambda_i$ be its nonzero part. Then
\begin{equation}\label{prop-Mac-e1}
P_{\lambda}\Big(\Big[\frac{1-q}{t-1}\big(m_1+m_2+\cdots+m_{\lambda_i-1}\big)+n_1+\cdots+n_{i-1}\Big];q,t\Big)=0,
\end{equation}
where the $m_i$ and the $n_i$ are single-letter alphabets.
\end{prop}
\begin{proof}
By \eqref{e-Mac4}
\begin{multline}\label{prop-Mac-e2}
P_{\lambda}\Big(\Big[\frac{1-q}{t-1}\big(m_1+m_2+\cdots+m_{\lambda_i-1}\big)+n_1+\cdots+n_{i-1}\Big];q,t\Big) \\
=\sum_{\mu\subset \lambda}P_{\lambda/\mu}\Big(\Big[\frac{1-q}{t-1}\big(m_1+m_2+\cdots+m_{\lambda_i-1}\big)\Big];q,t\Big)P_{\mu}\big([n_1+\cdots+n_{i-1}];q,t\big).
\end{multline}
By Lemma~\ref{lem-Mac-1}, the summand in the sum of \eqref{prop-Mac-e2} can be written as
\[
(-1)^{|\lambda|-|\mu|}Q_{\lambda'/\mu'}\Big(\big[m_1+m_2+\cdots+m_{\lambda_i-1}\big];t,q\Big)
P_{\mu}\big([n_1+\cdots+n_{i-1}];q,t\big).
\]
If $\ell(\mu)\geq i$ then $P_{\mu}\big([n_1+\cdots+n_{i-1}];q,t\big)=0$ by \eqref{e-Mac5}.
On the other hand, if $\ell(\mu)<i$ then $\mu_i=0$ and
\[
\lambda_i-\mu_i=\lambda_i>\mathrm{Car}(m_1+\cdots+m_{\lambda_i-1})=\lambda_i-1,
\]
where $\mathrm {Car}(X)$ means the cardinality of the alphabet $X$.
It follows that
\[
Q_{\lambda'/\mu'}\Big(\big[m_1+m_2+\cdots+m_{\lambda_i-1}\big];t,q\Big)=0
\]
by \eqref{e-Mac6} with $(q,t,\lambda,\mu)\mapsto (t,q,\lambda',\mu')$.
\end{proof}

Define the modified complete symmetric function $g_n(X;q,t)$ by its generating function
\begin{equation}\label{defi-g}
\prod_{i\geq 1}\frac{(tx_iy;q)_{\infty}}{(x_iy;q)_{\infty}}=\sum_{n\geq 0}g_n(X;q,t)y^n,
\end{equation}
and for any partition $\lambda=(\lambda_1,\lambda_2,\dots)$ define
\[
g_{\lambda}=g_{\lambda}(X;q,t)=\prod_{i\geq 1}g_{\lambda_i}(X;q,t).
\]
In fact $g_n$ can be written in plethystic notation
\begin{equation}\label{modi-complete}
g_n(X;q,t)=h_n\Big[\frac{1-t}{1-q}X\Big].
\end{equation}
In this paper, we only need the $t=q^c$ case of \eqref{modi-complete}. This particular case can be easily seen by comparing the generating function of $h_n$ and $g_n$ in \eqref{e-gfcomplete} and \eqref{defi-g} respectively.
For $\lambda=(r)$,
\begin{equation}\label{relation-P-g}
P_{(r)}=\frac{(q)_r}{(t)_r}g_r.
\end{equation}
See \cite[VI, (4.9)]{Mac95}.

The next result is called the Pieri formula for Macdonald polynomials \cite[VI, (6.24)]{Mac95}.
\begin{prop}\label{prop-Pieri}
Let $\mu$ be a partition. For $r$ a positive integer,
\begin{equation}
P_{\mu}g_r=\sum_{\lambda}\varphi_{\lambda/\mu}P_{\lambda},
\end{equation}
where the sum is over partitions $\lambda$ such that $\lambda/\mu$ is a horizontal $r$-strip.
\end{prop}
The explicit expression for $\varphi_{\lambda/\mu}$ is given in \cite[VI, (6.24)]{Mac95}.
We do not need the explicit form of $\varphi_{\lambda/\mu}$ in this paper.

The symmetric functions $g_{\lambda}$ form a basis of
$\Lambda_{\mathbb{F}}$ \cite[VI, (2.19)]{Mac95}.
Then the Macdonald polynomials can be written as a linear combination of
the modified complete symmetric functions.
The formula is referred as the generalized Jacobi-Trudi expansions for Macdonald polynomials.
We state that in the next theorem.
\begin{thm}\cite{Lassalle}\label{thm-Lassalle}
Let $\lambda$ and $\mu$ be partitions with length at most $n$. Then
\begin{equation}\label{expand-P}
P_{\lambda}=\sum_{\mu\geq \lambda}c_{\mu}g_{\mu}.
\end{equation}
\end{thm}
Note that the explicit expression for $c_{\mu}$ is quite complicated, we refer the reader to
\cite[Theorem 3]{Lassalle}. When $q=t$, \eqref{expand-P} reduces to the Jacobi-Trudi identity
\[
s_{\lambda}=\det(h_{\lambda_i-i+j})_{1\leq i,j\leq n},
\]
where $s_{\lambda}$ is the Schur function.

We can express the skew Macdonald polynomials as (usual) Madonald polynomials.
\begin{prop}\label{prop-Mac-skew}
For  partitions $\lambda$ and $\mu$ such that $\mu\subset \lambda$ and $\ell(\mu)<\ell(\lambda)$,
let
\begin{equation}\label{coeff-r}
P_{\lambda/\mu}=\sum_{\nu}r_{\mu \nu}^{\lambda}\cdot P_{\nu},
\end{equation}
where $\nu$ is over all partitions such that $|\nu|=|\lambda|-|\mu|$.
Then $r_{\mu \nu}^{\lambda}=0$ for $\nu_1<\lambda_{\ell(\mu)+1}$, or equivalently,
$\nu_1\geq \lambda_{\ell(\mu)+1}$.
\end{prop}
\begin{proof}
Substituting \eqref{PQ} into \eqref{skewQ1} yields
\[
Q_{\lambda/\mu}=\sum_{\nu}f_{\mu \nu}^{\lambda}Q_{\nu}
=\sum_{\nu}f_{\mu \nu}^{\lambda}b_{\nu}P_{\nu}.
\]
Together with $Q_{\lambda/\mu}=b_{\lambda}b_{\mu}^{-1}P_{\lambda/\mu}$ in \eqref{skewQP}, we have
\[
P_{\lambda/\mu}
=\sum_{\nu}f_{\mu \nu}^{\lambda}b_{\mu}b_{\lambda}^{-1}b_{\nu}P_{\nu}.
\]
Comparing this with \eqref{coeff-r}, we have
\[
r_{\mu \nu}^{\lambda}=b_{\mu}b_{\nu}b_{\lambda}^{-1}f_{\mu \nu}^{\lambda}.
\]
Hence, to prove the proposition it suffices to prove that a family of $q,t$-Littlewood--Richardson coefficients are zeros. That is
\begin{equation}\label{Mac-vanish1}
f_{\mu \nu}^{\lambda}=\langle Q_{\lambda},P_{\mu}P_{\nu}\rangle=0
\end{equation}
for $\nu_1<\lambda_{\ell(\mu)+1}$, $\mu\subset \lambda$ and $\ell(\mu)<\ell(\lambda)$.

By \eqref{expand-P}, we can write
\[
P_{\mu}=\sum_{\omega\geq \mu}c_{\omega}g_{\omega}.
\]
Using this and by the linearity of the $q,t$-Hall scalar product, we have
\begin{equation}\label{Mac-vanish2}
f_{\mu \nu}^{\lambda}=\langle Q_{\lambda},P_{\mu}P_{\nu}\rangle=\sum_{\omega\geq \mu}c_{\omega}
\langle Q_{\lambda},P_{\nu}g_{\omega}\rangle.
\end{equation}
We can get $\ell(\omega)\leq \ell(\mu)$ by $\omega\geq \mu$.
Consequently, $\ell(\omega)<\ell(\lambda)$.
Then we prove \eqref{Mac-vanish1} by showing that each term in the sum of \eqref{Mac-vanish2}
\begin{equation}\label{Mac-vanish3}
\langle Q_{\lambda},P_{\nu}g_{\omega}\rangle=0
\end{equation}
for $\nu_1<\lambda_{\ell(\mu)+1}$ and $\ell(\omega)<\ell(\lambda)$.
Repeatedly using the Pieri formulas for Macdonald polynomials (Proposition~\ref{prop-Pieri}) for
$\ell(\omega)$ times, we have
\[
P_{\nu}g_{\omega}=P_{\nu}\cdot g_{\omega_1}\cdot g_{\omega_2}\cdots
=\sum_{\lambda^*}\overline{c}_{\lambda^*}P_{\lambda^*},
\]
where $\overline{c}_{\lambda^*}$ is the coefficient for $P_{\lambda^*}$.
Since $\nu_1<\lambda_{\ell(\mu)+1}$ and $\ell(\omega)\leq \ell(\mu)<\ell(\lambda)$,
none of the partitions $\lambda^*$ can be $\lambda$.
(It needs at least $(\ell(\mu)+1)$-th use of the Pieri formula).
It follows that
\[
\langle Q_{\lambda},P_{\nu}g_{\omega}\rangle
=\sum_{\lambda^*}\overline{c}_{\lambda^*}\langle Q_{\lambda},P_{\lambda^*}\rangle=0. \qedhere
\]
\end{proof}

By Lemma~\ref{lem-positiveroots}, Theorem~\ref{thm-Lassalle} and Proposition~\ref{prop-Mac-skew},
we can obtain the next vanishing proposition for Macdonald polynomials.
\begin{prop}\label{prop-Mac-vanish-2}
Let $v=(v_1,\dots,v_n)\in \mathbb{Z}^n$, $\lambda$ and $\mu$ be partitions such that
$\mu\subset \lambda$, $\ell(\mu)<\ell(\lambda)$, $|\lambda|-|\mu|=|v|$ and
$\lambda_{\ell(\mu)+1}>\max\{v_i\mid i=1,\dots,n\}$.
Then
\begin{equation}
\CT_{x}x^{-v}P_{\lambda/\mu}(x;q,q^c)\prod_{1\leq i<j\leq n}
\Big(\frac{x_{i}}{x_{j}}\Big)_c\Big(\frac{x_{j}}{x_{i}}q\Big)_c=0.
\end{equation}
\end{prop}
\begin{proof}
By Proposition~\ref{prop-Mac-skew},
\[
P_{\lambda/\mu}=\sum_{\nu}r_{\mu \nu}^{\lambda}\cdot P_{\nu},
\]
where $\nu$ is over all partitions such that $|\nu|=|\lambda|-|\mu|=|v|$
and $\nu_1\geq \lambda_{\ell(\mu)+1}$.
By Theorem~\ref{thm-Lassalle}, every $P_{\nu}$ can be written as a linear combination of
$g_{\omega}$ for $|\omega|=|\nu|$ and $\omega\geq \nu$. Then, to prove the proposition,
it suffices to show that
\begin{multline}\label{e-g-vanish}
\CT_x x^{-v}g_{\omega}(x;q,q^c)\prod_{1\leq i<j\leq n}
\Big(\frac{x_{i}}{x_{j}}\Big)_c\Big(\frac{x_{j}}{x_{i}}q\Big)_c\\
=\CT_x x^{-v}h_{\omega}\Big[\sum_{i=1}^n\frac{1-q^c}{1-q}x_i\Big]\prod_{1\leq i<j\leq n}
\Big(\frac{x_{i}}{x_{j}}\Big)_c\Big(\frac{x_{j}}{x_{i}}q\Big)_c
=0
\end{multline}
for $\omega$ a partition such that $|\omega|=|v|$ and $\omega_1>\max\{v\}$.
Here the conditions for $\omega$ follows from $\omega\geq \nu$, $\nu_1\geq \lambda_{\ell(\mu)+1}$
and $\lambda_{\ell(\mu)+1}>\max\{v_i\mid i=1,\dots,n\}$.
By Lemma~\ref{lem-positiveroots}, we can conclude that
\eqref{e-g-vanish} holds.
\end{proof}

For positive integers $s$ and $n$, let $u:=(u_1,u_2,\dots,u_s)$ be a vector of integers such that $1\leq u_1<u_2<\cdots<u_s\leq n$. For a polynomial $f(x)=f(x_1,\dots,x_n)$ and $c_i\in \mathbb{Q}(q)$, we define a substitution $T_u\big(f(x)\big)$ by replacing $x_{u_i}=c_ix_{u_s}$ for $i=1,\dots,s-1$ in $f(x)$.
Then we obtain the next proposition which concerns the degrees of the Macdonald polynomials.
\begin{prop}\label{prop-Macdeg}
For a partition $\lambda$,
the degree in $x_{u_s}$ of $T_u\big(P_{\lambda}(x;q,t)\big)$ is at most $\lambda_1+\cdots+\lambda_s$.
\end{prop}
\begin{proof}
By \eqref{Mac-m}, the Macdonald polynomial $P_{\lambda}$ is a linear combination of those monomial symmetric functions $m_{\mu}$ for $\mu\leq \lambda$.
Hence, to prove the proposition, it suffices to prove that for $\mu\leq \lambda$,
the degree in $x_{u_s}$ of $T_u\big(m_{\mu}(x)\big)$ is no more than $\lambda_1+\cdots+\lambda_s$.
By the definition of $m_{\mu}$, it is clear that
this degree is no more than $\mu_1+\cdots+\mu_s$,
which is no more than $\lambda_1+\cdots+\lambda_s$ by
$\mu\leq \lambda$.
\end{proof}

\section{Polynomiality and rationality}\label{sec-poly}

In this section, we complete the Steps (1)-(2)
in the outline of the proof in the introduction.

We begin with the next polynomiality lemma.
\begin{lem}\label{lem1}
Let $L(x_1,\dots,x_n)$ be an arbitrary Laurent polynomial independent of $a$ and $x_0$.
Then, for an integer $t$ such that $t\leq nb$, the constant term
\begin{equation}\label{p1}
\CT_x x_0^{t} L(x_1,\dots,x_n) \prod_{i=1}^n (x_0/x_i)_a(qx_i/x_0)_b
\end{equation}
is a polynomial in $q^a$ of degree at most $nb-t$.
Moreover, if $t>nb$, then the constant term \eqref{p1} vanishes.
\end{lem}
\begin{proof}
The proof is almost the same as that of \cite[Lemma 2.2]{XZ}.
\end{proof}
By Lemma~\ref{lem1}, it is easy to conclude that $A_n(a,b,c,\lambda,\mu)$ is a polynomial in $q^a$.
\begin{cor}\label{cor-poly-BC}
The constant term $A_n(a,b,c,\lambda,\mu)$ is a polynomial in $q^a$ of degree at most $nb+|\lambda|+|\mu|$, assuming all the parameters but $a$ are fixed.
\end{cor}
\begin{proof}
By \eqref{e-Mac4} and \eqref{e-homo-sym}, we can write $A_n(a,b,c,\lambda,\mu)$ as
\begin{equation}\label{poly-B}
\sum_{\nu\subset \mu}P_{\nu}\big(\big[\frac{q^{c-b-1}-q^a}{1-q^c}\big];q,q^c\big)
\CT_xx_0^{-|\lambda|-|\mu|+|\nu|}
\prod_{i=1}^n (x_0/x_i)_a(qx_i/x_0)_b
\times L(x),
\end{equation}
where
\[
L(x)=
P_{\mu/\nu}(x;q,q^c)P_{\lambda}(x;q,q^c)
\prod_{1\leq i<j\leq n}\Big(\frac{x_{i}}{x_{j}}\Big)_c\Big(\frac{x_{j}}{x_{i}}q\Big)_c.
\]
We write
\begin{align*}
P_{\nu}\big(\big[\frac{q^{c-b-1}-q^a}{1-q^c}\big];q,q^c\big)
&=\sum_{\lambda\geq \nu}c_{\lambda}g_{\lambda}\Big(\Big[\frac{q^{c-b-1}-q^a}{1-q^c}\Big];q,q^c\Big)
\quad &\text{by \eqref{expand-P}}\\
&=\sum_{\lambda\geq \nu}c_{\lambda}h_{\lambda}\Big[\frac{q^{c-b-1}-q^a}{1-q}\Big]
\quad &\text{by \eqref{modi-complete}}\\
&=\sum_{\lambda\geq \nu}c_{\lambda}q^{(c-b-1)|\lambda|}h_{\lambda}\Big[\frac{1-q^{a+b+1-c}}{1-q}\Big]
\quad &\text{by \eqref{e-homo-sym}}\\
&=\sum_{\lambda\geq \nu}c_{\lambda}q^{(c-b-1)|\lambda|}
\prod_{i\geq 1}\frac{(q^{a+b+1-c})_{\lambda_i}}{(q)_{\lambda_i}}.
\end{align*}
The right-most equality holds by $h_r[(1-z)/(1-q)]=(z)_r/(q)_r$, see e.g., \cite[page 27]{Mac95}.
Then one can see that $P_{\nu}\big[\frac{q^{c-b-1}-q^a}{1-q^c}\big]$
is a polynomial in $q^a$ of degree at most $|\nu|$. Together with the fact that
the constant term in \eqref{poly-B} is a polynomial in $q^a$ of degree at most $nb+|\lambda|+|\mu|-|\nu|$
by Lemma~\ref{lem1}, we conclude that $A_n(a,b,c,\lambda,\mu)$ is a polynomial in $q^a$ of degree at most $nb+|\lambda|+|\mu|$.
\end{proof}

The following rationality result, which is implicitly due to
Stembridge \cite{stembridge1987}, as can be seen from the proof. One can also see this result in
\cite[Proposition~3.1]{XZ} and \cite[Lemma~7.5]{KNPV}.
The $q=1$ case of this result
is the equal parameter case of \cite[Proposition 2.4]{Gessel-Lv-Xin-Zhou2008}.
\begin{prop}\label{prop-rationality}
Let $\alpha=(\alpha_1,\dots,\alpha_n)\in \mathbb{Z}^{n}$ such that $|\alpha|=0$. Then
\begin{align}
\CT_{x} x^{\alpha}\prod_{1\leq i<j\leq n}
\Big(\frac{x_{i}}{x_{j}}\Big)_{c}\Big(\frac{x_{j}}{x_{i}}q\Big)_{c}=\frac{(q)_{nc}}{(q)_{c}^{n}}\cdot
R_n(q^c;q,\alpha),
\end{align}
where $R_n(q^c;q,\alpha)$ is a rational function in $q^c$ and $q$.
\end{prop}

By Proposition~\ref{prop-rationality} we have the next
straightforward consequence.
\begin{cor}\label{cor-rationality}
Let $d$ be a nonnegative integer independent of $c$ and $H$ be a homogeneous Laurent polynomial in $x_0,x_1,\dots,x_n$ of degree 0. If an expression for
\[
\CT_x H\prod_{1\leq i<j\leq n}
\Big(\frac{x_{i}}{x_{j}}\Big)_{c}\Big(\frac{x_{j}}{x_{i}}q\Big)_{c}
\]
holds for $c\geq d$, then it holds for all nonnegative integers $c$.
\end{cor}
\begin{proof}
By Proposition~\ref{prop-rationality}, it exists a rational function $R_n(q^c;q,H)$ such that
\[
\CT_x H\prod_{1\leq i<j\leq n}
\Big(\frac{x_{i}}{x_{j}}\Big)_{c}\Big(\frac{x_{j}}{x_{i}}q\Big)_{c}
=\frac{(q)_{nc}}{(q)_{c}^{n}}\cdot R_n(q^c;q,H).
\]
Then
\begin{equation}\label{rational-p1}
\frac{(q)_c^n}{(q)_{nc}}\CT_x H\prod_{1\leq i<j\leq n}
\Big(\frac{x_{i}}{x_{j}}\Big)_{c}\Big(\frac{x_{j}}{x_{i}}q\Big)_{c}
=R_n(q^c;q,H).
\end{equation}
It follows that if \eqref{rational-p1} holds for $c\geq d$ for some nonnegative integer $d$,
then it holds for all nonnegative integers $c$ since the both sides are rational functions in $q^c$.
\end{proof}

\section{Preliminaries for determination of the roots}\label{sec-prepare}

In this section, we present several results that are essential to determine the roots of
$A_{n}(a,b,c,\lambda,\mu)$.

\subsection{The key Lemmas}\label{subsec-1}

In this section, we obtain two key lemmas. The first one --- Lemma~\ref{cor-key} blew
--- is to discuss the distribution of some integers.
This kind of discussion in order to evaluate some constant terms may originate from
\cite[Lemma 4.2]{GX}. Each time when this lemma was extended,  some further (maybe harder) constant terms can be evaluated. See \cite{LXZ, Zhou} in dealing with the $q$-Dyson style constant terms for example. Lemma~\ref{cor-key} is a general extension of \cite[Lemma~4.2]{XZ}.
The $t=1$ case of Lemma~\ref{cor-key} corresponds to \cite[Lemma~4.2]{XZ}.
We will introduce the second key lemma blew.
\begin{lem}\label{cor-key}
For $s$ a positive integer, let $k_1,\dots,k_s$ and $t$ be nonnegative integers such that $1\leq k_{i}\leq (s-1)c+b+t$ for $1\leq i\leq s$.
Then, at least one of the following holds:
\begin{enumerate}
\item $1\leq k_i\leq b$ for some $i$ with $1\leq i\leq s$;
\item $-c\leq k_i-k_j\leq c-1$ for some $(i,j)$ such that $1\leq i<j\leq s$;
\item there exists a permutation $w\in\mathfrak{S}_s$ and nonnegative integers $t_1,\dots,t_s$
such that
\begin{subequations}\label{e-k1}
\begin{equation}
k_{w(1)}=b+t_1,
\end{equation}
and
\begin{equation}
k_{w(j)}-k_{w(j-1)}=c+t_j \quad \text{for $2\leq j\leq s$.}
\end{equation}
\end{subequations}
Here the $t_j$ satisfy
\begin{equation}\label{e-range-t}
1\leq \sum_{j=1}^{s}t_j\leq t,
\end{equation}
$w(0):=0$, and $t_j>0$ if $w(j-1)<w(j)$ for $1\leq j\leq s$.
In particular, if $t=1$ then
\begin{equation}\label{cor-key-1}
k_i=(s-i)c+b+1
\end{equation}
for $i=1,\dots,s$.
\end{enumerate}
\end{lem}
\begin{proof}
We prove the lemma by showing that if (1) and (2) fail then (3) must hold.

Assume that (1) and (2) are both false.
Then we construct a weighted tournament $T$ on the complete graph
on $s$ vertices, labelled $1,\dots,s$, as follows.
For the edge $(i,j)$ with $1\leq i<j\leq s$, we draw an arrow from $j$ to $i$ and attach a weight $c$ if $k_i-k_j\ge c$.
If, on the other hand, $k_i-k_j\le -c-1$ then we draw an arrow
from $i$ to $j$ and attach the weight $c+1$.
Note that the weight of each edge of a tournament is nonnegative.

We call a directed edge from $i$ to $j$ ascending if $i<j$.
It is immediate from our construction that
(i) the weight of the edge $i\to j$ is less than or equal $k_j-k_i$, and
(ii) the weight of an ascending edge is positive.

We will use (i) and (ii) to show that any of the above-constructed
tournaments is acyclic and hence transitive.
As consequence of (i), the weight of a directed path from $i$ to $j$ in $T$,
defined as the sum of the weights of its edges, is at most $k_j-k_i$.
Proceeding by contradiction, assume that $T$ contains a cycle $C$.
By the above, the weight of $C$ must be non-positive, and hence $0$.
Since $C$ must have at least one ascending edge, which by (ii) has positive
weight, the weight of $C$ is positive, a contradiction.

Since each $T$ is transitive, there is exactly one directed Hamilton path $P$ in $T$,
corresponding to a total order of the vertices.
Assume $P$ is given by
\[
P=w(1)\rightarrow w(2)\rightarrow\cdots\rightarrow w(s-1)\rightarrow w(s),
\]
where we have suppressed the edge weights.
Then
\[
k_{w(s)}-k_{w(1)}\ge (s-1)c,
\]
and thus
\begin{equation}\label{e-contradiction}
k_{w(s)}\ge k_{w(1)}+(s-1)c \ge b+1+(s-1)c.
\end{equation}
Together with the assumption that $k_{w(s)}\leq (s-1)c+b+t$
this implies that $P$ has at most $t-1$ ascending edges.
Let $t_1,\dots,t_s$ be nonnegative integers such that \eqref{e-k1} holds.
When $j=1$ this gives $k_{w(1)}=b+t_1$.
Since (1) does not hold, $k_{w(1)}\geq b+1$, so that $t_1>0$.
For $2\leq j\leq s$, if $w(j-1)\to w(j)$ is an ascending edge, then $t_j$ is a positive integer.
That is, for $2\leq j\leq s$ if $w(j-1)<w(j)$ then $t_j>0$.
Set $k_0:=0$.
Since
\begin{equation*}
\sum_{j=1}^s (k_{w(j)}-k_{w(j-1)})=k_{w(s)}
=b+(s-1)c+\sum_{j=1}^s t_j \leq b+(s-1)c+t,
\end{equation*}
we have $\sum_{j=1}^s t_j\leq t$.
Together with the fact that $t_1>0$ yields \eqref{e-range-t}.
This completes the proof of the assertion that (3) must hold if
both (1) and (2) fail.

If $t=1$, then none of the edges of $P$ can be ascending. That is $P=s\rightarrow s-1\rightarrow\cdots\rightarrow 1$.
Correspondingly, $k_i=(s-i)c+b+1$ for $i=1,\dots,s$.
\end{proof}

Lemma~\ref{cor-subs} allows us to combine the plethystic substitutions in symmetric functions with
the theory of iterated Laurent series. This discovery first appeared in \cite{Zhou}.
\begin{lem}\label{cor-subs}
For $s,t$ positive integers, let $k_1,\dots,k_s$ be nonnegative integers such that
$1\leq k_i\leq (s-1)c+b+t$ for $1\leq i\leq s$.
If the $k_i$ are such that (3) of Lemma~\ref{cor-key} holds, then
\begin{equation}\label{e-subs-2}
-\frac{q^{c-b-1}-q^a}{1-q}x_0-\sum_{i=1}^{s}\frac{1-q^{c}}{1-q}x_i
\bigg|_{\substack{-a=(s-1)c+b+t, \\[1pt] x_i=q^{k_s-k_i},\,0\leq i\leq s}}
=q^{n_1}+\dots+q^{n_{t-1}},
\end{equation}
where $\{n_1,\dots,n_{t-1}\}$ is a set of integers determined by $s,b,c$ and the $k_i$,
and we set $k_0:=0$.
In particular, set $q^{n_1}+\dots+q^{n_{t-1}}|_{t=1}:=0$.
\end{lem}
We remark that the set $\{n_1,\dots,n_{t}\}$ ($n_i\neq n_j$ for $1\le i<j\le t$) can be explicitly determined.
However, since the precise values of the $n_i$ are irrelavent in the following, we
have omitted them from the above statement.
Indeed, the important fact about the right-hand side is that, viewed as an
alphabet, has cardinality $t-1$.
\begin{proof}
Denote the left-hand side of \eqref{e-subs-2} by $L$.
Carrying out the substitutions
\[
a\mapsto -(s-1)c-b-t\quad\text{and}\quad  x_i\mapsto q^{k_s-k_i} \text{ for $0\leq i\leq s$}
\]
in
\[
-\frac{q^{c-b-1}-q^a}{1-q}x_0-\sum_{i=1}^{s}\frac{1-q^{c}}{1-q}x_i,
\]
we obtain
\begin{align}
L&=
\frac{q^{k_s}}{1-q}\bigg(q^{-(s-1)c-b-t}-q^{c-b-1}
-\sum_{i=1}^s\big(1-q^c\big)q^{-k_i}\bigg)\nonumber \\
&=\frac{q^{k_s}}{1-q}\bigg(q^{-(s-1)c-b-t}-q^{c-b-1}
-\sum_{i=1}^s(1-q^{c})q^{-k_{w(i)}}\bigg), \label{e-prop-subs1}
\end{align}
where $w\in \mathfrak{S}_s$ is any permutation such that \eqref{e-k1} holds.
By rearranging the terms in \eqref{e-prop-subs1}, it can be written as
\begin{multline*}
L=\frac{q^{k_s}}{1-q}\bigg(q^{c-k_{w(1)}}\big(1-q^{k_{w(1)}-b-1}\big)
+q^{-(s-1)c-b-t}\big(1-q^{(s-1)c+b+t-k_{w(s)}}\big)\\
+\sum_{i=2}^{s}q^{c-k_{w(i)}}\big(1-q^{-c+k_{w(i)}-k_{w(i-1)}}\big)
\bigg).
\end{multline*}
By \eqref{e-k1} and $k_i\leq (s-1)c+b+t$, we have
\begin{subequations}\label{s}
\begin{equation}\label{s1}
k_{w(1)}-b-1\in\mathbb{N},
\end{equation}
\begin{equation}\label{s2}
(s-1)c+b+t-k_{w(s)}\in\mathbb{N},
\end{equation}
and
\begin{equation}\label{s3}
-c+k_{w(i)}-k_{w(i-1)}\in\mathbb{N} \quad\text{for $i=2,3,\dots,s$},
\end{equation}
\end{subequations}
where $\mathbb{N}=\{0,1,2,\dots\}$.
Since $(1-q^n)/(1-q)=1+\dots+q^{n-1}$ for $n$ a positive integer and is 0 for $n=0$, we may conclude
that $L=q^{n_1}+\dots+q^{n_p}$,
where $p$ is given by
\begin{equation*}
p=k_{w(1)}-b-1+(s-1)c+b+t-k_{w(s)}+\sum_{i=2}^s\big(-c+k_{w(i)}-k_{w(i-1)}\big)=t-1.
\end{equation*}
This completes the proof.
\end{proof}

\subsection{The rational function $Q(d\Mid u;k)$}\label{subsec-2}
In this section, we discuss a kind of vanishing and recursive properties of the rational function $Q(d\Mid u;k)$ defined in \eqref{eq-Qrk} below.

Let
\begin{multline}\label{def-Q3}
Q(d)=x_0^{-|\lambda|-|\mu|}P_{\lambda}(x;q,q^c)
P_{\mu}\Big(\Big[\frac{q^{c-b-1}-q^{-d}}{1-q^c}x_0+\sum_{i=1}^nx_i\Big];q,q^c\Big)\\
\times\prod_{i=1}^{n}\frac{(qx_{i}/x_{0})_b}{(q^{-d}x_{0}/x_{i})_d}
\prod_{1\leq i<j\leq n}
\Big(\frac{x_{i}}{x_{j}}\Big)_c\Big(\frac{x_{j}}{x_{i}}q\Big)_c.
\end{multline}
It is clear that
\[
\CT_x Q(-a)=A_{n}(a,b,c,\lambda,\mu).
\]
Note that the above equation holds for all integers $a$ by Corollary~\ref{cor-poly-BC}.

For any rational function $F$ of $x_0, x_1, \dots, x_n$ and $s$ an integer such that $1\leq s\leq n$, and for
sequences of integers $k = (k_1, k_2, \dots, k_s)$ and $u = (u_1,
u_2, \dots, u_s)$ let $E_{u,k}F$ be the result of
replacing $x_{u_i}$ in $F$ with $x_{u_s}q^{k_s-k_i}$ for $i = 0,
1,\dots, s-1$, where we set $u_0 = k_0 = 0$. Then for $0 < u_1 <
u_2 <\dots < u_s \leq n$ and $1\leq k_i\leq d$, we define
\begin{equation}\label{eq-Qrk}
Q(d\Mid u;k):=Q(d\Mid u_1,\dots,u_s;k_1,\dots,k_s)=E_{u,k}
\bigg(Q(d)\prod_{i=1}^{s}(1-\frac{x_{0}}{x_{u_{i}}q^{k_{i}}})\bigg).
\end{equation}
Set $Q(d\Mid u;k)|_{s=0}:=Q(d)$.
Note that the product on the right hand side of \eqref{eq-Qrk}
cancels all the factors in the denominator of $Q$ that would be
taken to zero by $E_{u,k}$.

The numerator of $Q(d)$ ---
\[
x_0^{-|\lambda|-|\mu|}P_{\lambda}(x;q,q^c)
P_{\mu}\Big(\Big[\frac{q^{c-b-1}-q^{-d}}{1-q^c}x_0+\sum_{i=1}^nx_i\Big];q,q^c\Big)\\
\times\prod_{i=1}^{n}(qx_{i}/x_{0})_b
\prod_{1\leq i<j\leq n}
\Big(\frac{x_{i}}{x_{j}}\Big)_c\Big(\frac{x_{j}}{x_{i}}q\Big)_c
\]
is a Laurent polynomial in $x_0$ with degree $-|\lambda|$ at most.
The denominator of $Q(d)$ ---
\[
\prod_{i=1}^{n}\big(q^{-d}x_{0}/x_{i}\big)_{d}
\]
is of the form
\[
\prod_{r=1}^{nd} (1-c_r x_0/ x_{i_r})
\]
with degree $nd$ in $x_0$, where all the $i_r\neq 0$, and $c_r\neq c_v$ if $i_r=i_v$.
Thus, for $d$ a positive integer $Q(d)$ is a rational function of the form
\eqref{e-defF} with respect to $x_0$. Then we can apply Lemma~\ref{lem-prop}
to $Q(d)$ with respect to $x_0$ and obtain
\begin{equation}\label{Q3-1}
\CT_{x_0}Q(d)=\sum_{\substack{1\leq k_1\leq d\\1\leq u_1\leq n}}
Q(d\Mid u_1;k_1),
\end{equation}
where $Q(d\Mid u_1;k_1)$ is defined in \eqref{eq-Qrk} with $s=1$.
We can further apply Lemma~\ref{lem-prop} to each $Q(d\Mid u_1;k_1)$ with respect to $x_{u_1}$ if applicable, and get a sum.
Continue this operation until Lemma~\ref{lem-prop} does not apply to every summand. In other words, every summand can not be written as a sum by Lemma~\ref{lem-prop}. Finally we write
\begin{equation}\label{Q3-sum}
\CT_{x}Q(d)=\sum_{s\in T\subseteq \{1,\dots,n\}}\sum_{\substack{1\leq u_1<\cdots<u_s\leq n\\1\leq k_1,\dots,k_s\leq d}}
\CT_xQ(d\Mid u_1,\dots,u_s;k_1,\dots,k_s).
\end{equation}
We call this operation the Gessel--Xin operation to the rational function $Q(d)$, since it first appeared in \cite{GX}.

The following vanishing and recursive properties of $Q(d\Mid u;k)$ is crucial in the proof of our main Theorem.
\begin{prop}\label{cor-Q}
For $s$ an integer such that $1\leq s\leq n$, the rational functions $Q(d\Mid u;k)$ have the following properties:
\begin{enumerate}
\item If $d\leq (s-1)c+b$, then $Q(d\Mid u;k)=0$;

\item If $d>sc$ and $s\neq n$, then
\begin{equation}\label{cor-Q-e1}
\CT_{x_{u_s}}Q(d\Mid u;k)=
\begin{cases}\displaystyle
\sum_{\substack{u_s<u_{s+1}\leq n\\1\leq k_{s+1}\leq d}}
(d\Mid u_1,\dots,u_s,u_{s+1};k_1,\dots,k_s,k_{s+1}) \quad &\text{for $u_s<n$,}\\
0 \quad &\text{for $u_s=n$;}
\end{cases}
\end{equation}
\end{enumerate}
\end{prop}

\begin{proof}
(1)
If $d\leq (s-1)c+b$, then
since $1\leq k_i\leq d$ for $i=1,\dots,s$ we have
\[
1\leq k_i\leq (s-1)c+b \quad \text{for $i=1,\dots,s$}.
\]
By Lemma~\ref{cor-key} with the $t=0$ case,
its case (3) can not occur. Then either $1\leq k_i\leq b$ for some $i$ with $1\leq i\leq s$,
or $-c\leq k_i-k_j\leq c-1$ for some $(i,j)$ such that $1\leq i<j\leq s$.

If $1\leq k_i\leq b$ for some $i$, then $Q(d\Mid u;k)$ has the factor
\[
E_{u,k}\Big[\big(qx_{u_i}/x_0\big)_{b}\Big]=\big(q^{1-k_i}\big)_{b}=0.
\]

If $-c\leq k_i-k_j\leq c-1$ for some $(i,j)$, then $Q(d\Mid u;k)$ has the factor
\[
E_{u,k}\Big[\big(x_{u_i}/x_{u_j}\big)_c\big(qx_{u_j}/x_{u_i}\big)_c\Big],
\]
which is equal to
\[
E_{u,k}\Big[q^{\binom{c+1}{2}}(-x_{u_j}/x_{u_i})^c\big(q^{-c}x_{u_i}/x_{u_j}\big)_{2c}\Big]
=q^{\binom{c+1}{2}}(-q^{k_i-k_j})^c\big(q^{k_j-k_i-c}\big)_{2c}=0.
\]

(2)
We first show that $Q(d\Mid u;k)$ is of the form \eqref{e-defF} for $d>sc$.

Denote $U=\{u_1,u_2,\dots,u_s\}$. By the expression for $Q(d\Mid u;k)$ in \eqref{eq-Qrk}, the parts contribute to the degree in $x_{u_s}$
of the numerator of  $Q(d\Mid u;k)$ is
\begin{multline*}
E_{u,k}\bigg(x_0^{-|\lambda|-|\mu|}P_{\lambda}(x;q,q^c)
P_{\mu}\Big(\Big[\frac{q^{c-b-1}-q^{-d}}{1-q^c}x_0+\sum_{i=1}^nx_i\Big];q,q^c\Big)\bigg)\\
\times \prod_{\substack{i=1\\ i\notin U}}^n\prod_{j=1}^s
\big(q^{k_s-k_j+\chi(u_j>i)}x_{u_s}/x_i\big)_c,
\end{multline*}
which has degree at most $s(n-s)c$.
The parts contribute to the degree in $x_{u_s}$ of the denominator of $Q(d\Mid u;k)$ is
\[
\prod_{\substack{i=1\\ i\notin U}}^n
\big(q^{k_s-d}x_{u_s}/x_i\big)_d,
\]
which has degree $(n-s)d$.
If $d>sc$ then $Q(d\Mid u;k)$ is of the form \eqref{e-defF}.
Applying Lemma~\ref{lem-prop} gives
\[
\CT_{x_{u_s}}Q(d\Mid u;k)=
\begin{cases}\displaystyle
\sum_{\substack{u_s<u_{s+1}\leq n\\1\leq k_{s+1}\leq d}}
Q(d\Mid u_1,\dots,u_s,u_{s+1};k_1,\dots,k_s,k_{s+1}) \quad &\text{for $u_s<n$,}\\
0 \quad &\text{for $u_s=n$.}
\end{cases}
\]
This completes the proof.
\end{proof}

\subsection{The Laurent polynomiality of $Q(d\Mid u;k)$}\label{subsec-3}

In this section, we find that the rational function $Q(d\Mid u;k)$ can be written as
a Laurent polynomial under certain conditions.

We can write $Q(d\Mid u;k)$ as
\begin{equation}\label{defi-Q}
Q(d\Mid u;k)=H\times
\prod_{i=1}^s(q^{k_i-d})^{-1}_{d-k_i}(q)^{-1}_{k_i-1}
\times V \times L\\
\times\prod_{\substack{1\leq i<j\leq n\\i,j\notin U}}\big(x_i/x_j\big)_{c}
\big(qx_j/x_i\big)_{c},
\end{equation}
where
\[
H=E_{u,k}\bigg(x_{0}^{-|\lambda|-|\mu|}\times P_{\lambda}(x;q,q^c)
P_{\mu}\Big(\Big[\frac{q^{c-b-1}-q^{-d}}{1-q^c}x_0+\sum_{i=1}^nx_i\Big];q,q^c\Big)\bigg),
\]
\[
V=\prod_{i=1}^s(q^{1-k_i})_{b}\prod_{1\leq i<j\leq s}(q^{k_j-k_i})_c(q^{k_i-k_j+1})_c,
\]
and
\begin{equation}\label{e-defi-L}
L=\prod_{\substack{i=1\\ i\notin U}}^n
\frac{(q^{1-k_s}x_i/x_{u_s})_{b}}{\big(q^{k_s-d}x_{u_s}/x_i\big)_d}
\prod_{\substack{i=1\\ i\notin U}}^n\prod_{j=1}^s
\big(q^{k_j-k_s+\chi(i>u_j)}x_i/x_{u_s}\big)_c
\big(q^{k_s-k_j+\chi(u_j>i)}x_{u_s}/x_i\big)_c.
\end{equation}

\begin{prop}\label{cor-Laurent}
Let $1\leq d\leq sc+b$.
For $i=1,\dots,s$ the $k_i$ satisfy $1\leq k_i\leq d$, and
there exists a permutation $w\in\mathfrak{S}_s$ and nonnegative integers $t_1,\dots,t_s$
such that
\begin{equation}\label{cor-kw1}
k_{w(1)}=b+t_1,
\end{equation}
and
\begin{equation}\label{cor-kwj}
k_{w(j)}-k_{w(j-1)}=c+t_j \quad \text{for $2\leq j\leq s$.}
\end{equation}
Here the $t_j$ satisfy
\begin{equation}\label{cor-sumbound}
1\leq \sum_{j=1}^{s}t_j\leq c,
\end{equation}
$w(0):=0$, and $t_j>0$ if $w(j-1)<w(j)$ for $1\leq j\leq s$.
Then the rational function $L$ in \eqref{e-defi-L} can be written as a Laurent polynomial of the form
\begin{equation}\label{cor-L-Laurent}
x_{u_s}^{(n-s)(sc-d)}\prod_{\substack{i=1\\ i\notin U}}^n
\Big(p_i(x_i/x_{u_s})x_i^{d-sc}\Big),
\end{equation}
where the $p_i(z)$ are polynomials in $z$.
\end{prop}
Note that the conditions for the $k_i$ in this proposition is just the case (3) of Lemma~\ref{cor-key}
with $t=c$.
\begin{proof}
We can further write
\begin{align}
L&=\prod_{\substack{i=1\\ i\notin U}}^n
\frac{(q^{1-k_s}x_i/x_{u_s})_{b}}{(-x_{u_s}/x_i)^dq^{dk_s-\binom{d+1}{2}}\big(q^{1-k_s}x_i/x_{u_s}\big)_d}\nonumber \\
&\quad\times\prod_{\substack{i=1\\ i\notin U}}^n\prod_{j=1}^s
(-x_{u_s}/x_i)^cq^{\big(k_s-k_j+\chi(u_j>i)\big)c+\binom{c}{2}}
\big(q^{k_j-k_s+1-c-\chi(u_j>i)}x_i/x_{u_s}\big)_{2c}\nonumber\\
&=C\times(x_{u_s})^{(n-s)(sc-d)}\prod_{\substack{i=1\\ i\notin U}}^n
x_i^{d-sc}\frac{(q^{1-k_s}x_i/x_{u_s})_{b}\prod_{j=1}^s
\big(q^{k_j-k_s+1-c-\chi(u_j>i)}x_i/x_{u_s}\big)_{2c}}
{\big(q^{1-k_s}x_i/x_{u_s}\big)_d},\label{e-expr-L}
\end{align}
where
\[
C=(-1)^{(n-s)(sc-d)}\prod_{\substack{i=1\\ i\notin U}}^n
q^{\sum_{j=1}^s\big((k_s-k_j+\chi(u_j>i))c+\binom{c}{2}\big)-dk_s+\binom{d+1}{2}}.
\]

Let
\begin{align*}
S_0:=&\{1-k_s,2-k_s,\dots,d-k_s\},
\quad B:=\{1-k_s,2-k_2,\dots,b-k_s\}, \\
S_j:=&\{k_j-k_s+1-c-\chi(u_j>i),\dots,k_j-k_s+c-\chi(u_j>i)\}, \quad \text{for $j=1,\dots,s$}.
\end{align*}
By the expression for $L$ in \eqref{e-expr-L}, to prove the proposition it suffices to show that the factor
\[
\prod_{\substack{i=1\\ i\notin U}}^n \big(q^{1-k_s}x_i/x_{u_s}\big)_d
\]
in the denominator can be cancelled by the numerator
\[
(q^{1-k_s}x_i/x_{u_s})_{b}\prod_{j=1}^s
\big(q^{k_j-k_s+1-c-\chi(u_j>i)}x_i/x_{u_s}\big)_{2c}.
\]
Then $L$ is a Laurent polynomial,
rather than a rational function.
To do this, it is equivalent to showing that
\begin{equation}\label{set}
S_0\subseteq \bigcup_{j=1}^sS_j\bigcup B.
\end{equation}
Since $S_j\supseteq S'_j:=\{k_j-k_s+1-c,\dots,k_j-k_s+c-1\}$ for $j=1,\dots,s$,
\eqref{set} holds if we show that
\begin{equation}\label{set2}
S_0\subseteq \bigcup_{j=1}^sS'_j\bigcup B.
\end{equation}
We write the ordering of the $S'_j$ as $S'_{w(1)},S'_{w(2)},\dots,S'_{w(s)}$ using the permutation $w$ in the proposition.
To obtain \eqref{set2}, it is sufficient to prove
\begin{enumerate}
\item $b-k_s\geq k_{w(1)}-k_s-c$, i.e., $k_{w(1)}\leq b+c$;
\item $k_{w(s)}-k_s+c-1\geq d-k_s$, i.e., $k_{w(s)}\geq d-c+1$;
\item $k_{w(j-1)}-k_s+c-1\geq k_{w(j)}-k_s-c$, i.e.,
$k_{w(j)}-k_{w(j-1)}\leq 2c-1$ for $j=2,\dots,s$.
\end{enumerate}

To prove (1), we find the upper bound for $k_{w(1)}$.
Since $k_{w(1)}=b+t_1$ and $t_1\leq \sum_{j=1}^st_j\leq c$, we have
$k_{w(1)}\leq b+c$. Then (1) holds.

To prove (2), we find the lower bound for $k_{w_s}$.
Summing both sides of \eqref{cor-kwj} for $j=2,\dots,s$ gives
\[
k_{w(s)}=k_{w(1)}+(s-1)c+\sum_{j=2}^st_j.
\]
By \eqref{cor-kw1}
\[
k_{w(s)}\geq (s-1)c+b+\sum_{j=1}^st_j\geq (s-1)c+b+1.
\]
Here in the last inequality we use \eqref{cor-sumbound}.
Since $d\leq sc+b$, we have
\[
k_{w(s)}\geq (s-1)c+b+1=sc+b-c+1=d-c+1.
\]
Thus (2) holds.

If (3) fails for some $i\in \{2,\dots,s\}$, then
\[
k_{w(i)}-k_{w(i-1)}\geq 2c.
\]
Together with \eqref{cor-kwj} gives
\[
k_{w(s)}-k_{w(1)}=\sum_{j=2}^s\big(k_{w(j)}-k_{w(j-1)}\big)
\geq sc+\sum_{\substack{j=2\\j\neq i}}^st_j\geq sc.
\]
Since $k_{w(1)}=b+t_1\geq b+1$, we have
\[
k_{w(s)}\geq sc+b+1=d+1.
\]
This contradicts the assumption that all the $k_i\leq d$.
Then (3) holds.
\end{proof}

\section{Proof of Theorem~\ref{thm-AFLT}}\label{sec-A}

By the discussions in Section~\ref{sec-poly}, we have completed the first two routine steps mentioned in the introduction for determining
$A_n(a,b,c,\lambda,\mu)$.
In this section, we complete the proof of Theorem~\ref{thm-AFLT} by finishing the last two steps:
In Subsections~\ref{sec-A-roots1}--\ref{sec-A3-2}, we determine all the roots for $A_n(a,b,c,\lambda,\mu)$.
In Subsection~\ref{sec-A-addition}, we characterize the expression for $A_n(a,b,c,\lambda,\mu)$ at an addition point.
We assume $c>b+\lambda_1+\mu_1$ throughout this section.

\subsection{Determination of the roots $A_1$}\label{sec-A-roots1}

In this subsection, we will determine the roots $A_1$ defined in \eqref{Roots-A} for
$A_n(a,b,c,\lambda,\mu)$, that is the content of the next lemma.

\begin{lem}\label{lem-roots-A1}
The constant term $A_n(a,b,c,\lambda,\mu)$ vanishes for $a\in A_1$.
\end{lem}
\begin{proof}
By the definitions of $A_n(a,b,c,\lambda,\mu)$ and $Q(d)$ in \eqref{Defi-A} and \eqref{def-Q3} respectively, we have
\begin{equation}\label{e-relationAQ}
A_n(-d,b,c,\lambda,\mu)=\CT_x Q(d)
\end{equation}
for $d$ an integer. We then prove the lemma by showing that
\begin{equation}\label{e-A1-1}
\CT_x Q(d)=0 \quad \text{for $-d\in A_1$}.
\end{equation}
We prove \eqref{e-A1-1} by induction on $n-s$ that
\begin{equation}\label{e-A1-2}
\CT_x Q(d\Mid u;k)=0 \quad \text{for $-d\in A_1$},
\end{equation}
where $Q(d\Mid u;k)$ is defined in \eqref{eq-Qrk}. The $s=0$ case of \eqref{e-A1-2} corresponds to \eqref{e-A1-1}.

For $-d\in A_1$ we have $1\leq d\leq (n-1)c+b$.
Then, by the property (1) of Proposition~\ref{cor-Q} with $s=n$, we have $Q(d\Mid u;k)=0$ for
$-d\in A_1$ and $s=n$. Thus, the induction basis holds.
Now suppose $0\leq s<n$.
For $d\in A_1$, if the property (1) of Proposition~\ref{cor-Q} applies,
then $Q(d\Mid u;k)=0$;
Otherwise, the property (2) of Proposition~\ref{cor-Q} applies and \eqref{cor-Q-e1} holds.
Applying $\CT\limits_x$ to both sides of \eqref{cor-Q-e1} gives
\begin{equation*}
\CT_{x}Q(d\Mid u;k)
=\begin{cases}\displaystyle
\sum_{\substack{u_s<u_{s+1}\leq n\\1\leq k_{s+1}\leq d}}
\CT_xQ(d\Mid u_1,\dots,u_s,u_{s+1};k_1,\dots,k_s,k_{s+1})  &\text{for $u_s<n$,}\\[2mm]
0  &\text{for $u_s=n$.}
\end{cases}
\end{equation*}
By the induction hypothesis, every term in the above sum is zero, and so is the sum.
Therefore, we obtain \eqref{e-A1-2}. Consequently, \eqref{e-A1-1} holds.
\end{proof}

\subsection{Determination of the roots $A_2$ for $A_n(a,b,c,\lambda,\mu)$}\label{sec-A-roots2}

In this subsection, we will determine the roots $A_2$ defined in \eqref{Roots-A} for
$A_n(a,b,c,\lambda,\mu)$. We assume that $A_2\neq \emptyset$, i.e., $\lambda\neq 0$ in
this subsection.

For $i=1,\dots,\ell(\lambda)$, denote
\[
A_{2,i}=\{-(i-1)c+\lambda_{i}-1,-(i-1)c+\lambda_{i}-2,\dots,-(i-1)c\}.
\]
Then $A_2=\cup_{i=1}^{\ell(\lambda)}A_{2,i}$. Notice that under the assumption
$c>b+\lambda_1+\mu_1$, the elements of $A_{2,1}$ are nonnegative integers and
the elements of $A_{2,i}$ are negative integers for $i>1$.
We first consider the cases when $a\in A_{2,1}$.
\begin{lem}\label{lem-roots-A21}
The constant term $A_n(a,b,c,\lambda,\mu)$ vanishes for $a\in A_{2,1}$.
\end{lem}
\begin{proof}
If $a=0$ then
\begin{align}\label{e-A2-1}
A_n(0,b,c,\lambda,\mu)=\CT_x x_0^{-|\lambda|-|\mu|}&P_{\lambda}(x;q,q^c)
P_{\mu}\Big(\Big[\frac{q^{c-b-1}-1}{1-q^c}x_0+\sum_{i=1}^nx_i\Big];q,q^c\Big)\\
&\times
\prod_{i=1}^{n}
\Big(\frac{qx_{i}}{x_{0}}\Big)_b\prod_{1\leq i<j\leq n}
\Big(\frac{x_{i}}{x_{j}}\Big)_c\Big(\frac{x_{j}}{x_{i}}q\Big)_c.\nonumber
\end{align}
Since the degree in $x_0$ of $P_{\mu}\big(\big[\frac{q^{c-b-1}-1}{1-q^c}x_0+\sum_{i=1}^nx_i\big];q,q^c\big)$ is at most $|\mu|$ (In fact, it is at most $\mu_1$), the right-hand side of \eqref{e-A2-1}
is a Laurent polynomial in $x_0$ with degree at most $-|\lambda|<0$.
Hence, $A_n(0,b,c,\lambda,\mu)=0$.
For $a\in A_{2,1}\setminus \{0\}$, expanding
\[
\prod_{i=1}^{n}\Big(\frac{x_{0}}{x_{i}}\Big)_a\Big(\frac{qx_{i}}{x_{0}}\Big)_b
P_{\mu}\Big(\Big[\frac{q^{c-b-1}-q^a}{1-q^c}x_0+\sum_{i=1}^nx_i\Big];q,q^c\Big)
\]
as a sum of monomials and extracting the constant term with respect to $x_0$
in $A_n(a,b,c,\lambda,\mu)$ gives
\begin{equation}\label{e-A2-2}
A_n(a,b,c,\lambda,\mu)
=\sum_{v}\CT_x c_vx^{-v}P_{\lambda}(x;q,q^c)
\prod_{1\leq i<j\leq n}
\Big(\frac{x_{i}}{x_{j}}\Big)_c\Big(\frac{x_{j}}{x_{i}}q\Big)_c,
\end{equation}
where the sum is over a finite set of $v=(v_1,\dots,v_n)\in \mathbb{Z}^n$ such that $|v|=|\lambda|$ and $\max\{v\}\leq a<\lambda_1$. If $A_{2,1}\setminus \{0\}\neq \emptyset$,
then $\lambda_1>1$.
By Proposition~\ref{prop-Mac-vanish-2} with $\mu=0$, every constant term in the sum of \eqref{e-A2-2} vanishes. Then $A_n(a,b,c,\lambda,\mu)=0$ for $a\in A_{2,1}\setminus \{0\}$.
Together with the fact that $A_n(0,b,c,\lambda,\mu)=0$, we conclude that
$A_n(a,b,c,\lambda,\mu)=0$ for $a\in A_{2,1}$.
\end{proof}

Now we consider the cases when $a\in A_{2,i}$ for $i>1$.
\begin{lem}\label{lem-roots-A2i}
Let $i\in \{2,\dots,\ell(\lambda)\}$ be a fixed integer.
Then the constant term $A_n(a,b,c,\lambda,\mu)$ vanishes for $a\in A_{2,i}$.
\end{lem}
\begin{proof}
By \eqref{e-relationAQ}, we show that
\begin{equation}\label{e-A2-3}
\CT_x Q(d)=0
\end{equation}
for $-d\in A_{2,i}$.
If $-d\in A_{2,i}$ then $(i-1)c-\lambda_i+1\leq d\leq (i-1)c$.
Applying the Gessel--Xin operation and using \eqref{Q3-sum},
\begin{equation}\label{Q3-sum-i}
\CT_{x}Q(d)=\sum_{s\in T\subseteq \{1,\dots,n\}}\sum_{\substack{1\leq u_1<\cdots<u_s\leq n\\1\leq k_1,\dots,k_s\leq d}}
\CT_xQ(d\Mid u_1,\dots,u_s;k_1,\dots,k_s).
\end{equation}
Note that the $s$ are maximal by the Gessel--Xin operation.
We show that the $s$ in \eqref{Q3-sum-i} can only be $i-1$, otherwise the summand vanishes.
If $s\geq i$, then $d\leq (i-1)c\leq (s-1)c+b$. By (1) of Proposition~\ref{cor-Q}, the summand
in \eqref{Q3-sum-i} vanishes.
If $s\leq i-2$, then $d\geq (i-1)c-\lambda_i+1\geq (s+1)c-\lambda_i+1\geq sc+1$ by the assumption
$c>b+\lambda_1+\mu_1$. It is clear that $s\neq n$ since $s\leq i-2\leq \ell(\lambda)-2\leq n-2$.
By (2) of Proposition~\ref{cor-Q}, the summand
in \eqref{Q3-sum-i} is either zero or can be written as a sum (against that $s$ is maximal).
Hence, the $s$ in \eqref{Q3-sum-i} can only be $i-1$ and the equation reduces to
 \begin{equation}\label{Q3-sum-i-2}
\CT_{x}Q(d)=\CT_x\sum_{\substack{1\leq u_1<\cdots<u_{i-1}\leq n\\1\leq k_1,\dots,k_{i-1}\leq d}}
Q(d\Mid u_1,\dots,u_{i-1};k_1,\dots,k_{i-1}).
\end{equation}
Denote by $u^{(i-1)}:=(u_1,\dots,u_{i-1})$ and  $k^{(i-1)}:=(k_1,\dots,k_{i-1})$.
We will prove that every $\CT\limits_x Q(d\Mid u^{(i-1)};k^{(i-1)})$ vanishes for $-d\in A_{2,i}$.

Since $d\geq (i-1)c-\lambda_i+1$ and $c>b+\lambda_1+\mu_1$,
we have $d\geq (i-2)c+b+1+\mu_1+(\lambda_1-\lambda_i)\geq (i-2)c+b+1$.
Then we can write $d=(i-2)c+b+t$ for a positive integer $t$.
It follows that $1\leq k_j\leq (i-2)c+b+t$ for $j=1,\dots,i-1$. By Lemma~\ref{cor-key} with $s=i-1$, at least one of the following holds:
\begin{enumerate}
\item[(i)] $1\leq k_j\leq b$ for some $j$ with $1\leq j\leq i-1$;
\item[(ii)] $-c\leq k_r-k_l\leq c-1$ for some $(r,l)$ such that $1\leq r<l\leq i-1$;
\item[(iii)] there exists a permutation $w\in\mathfrak{S}_{i-1}$ and nonnegative integers $t_1,\dots,t_{i-1}$
such that
\begin{subequations}
\begin{equation*}
k_{w(1)}=b+t_1,
\end{equation*}
and
\begin{equation*}
k_{w(j)}-k_{w(j-1)}=c+t_j \quad \text{for $2\leq j\leq i-1$.}
\end{equation*}
\end{subequations}
Here the $t_j$ satisfy
\begin{equation*}
\sum_{j=1}^{i-1}t_j\leq t,
\end{equation*}
$w(0):=0$, and $t_j>0$ if $w(j-1)<w(j)$ for $1\leq j\leq i-1$.
\end{enumerate}
If (i) holds, then $Q(d\Mid u^{(i-1)};k^{(i-1)})$ has the factor
\[
E_{u,k}\Big[\big(qx_{u_j}/x_0\big)_b\Big]=(q^{1-k_j})_b=0.
\]
If (ii) holds, then $Q(d\Mid u^{(i-1)};k^{(i-1)})$ has the factor
\[
E_{u,k}\Big[\big(x_{u_r}/x_{u_l}\big)_c\big(qx_{u_l}/x_{u_r}\big)_c\Big],
\]
which is equal to
\[
E_{u,k}\Big[q^{\binom{c+1}{2}}(-x_{u_l}/x_{u_r})^c\big(q^{-c}x_{u_r}/x_{u_l}\big)_{2c}\Big]
=q^{\binom{c+1}{2}}(-q^{k_r-k_l})^c(q^{k_l-k_r-c})_{2c}=0.
\]
For $-d\in A_{2,i}$, we have $(i-1)c-\lambda_i+1\leq d\leq (i-1)c\leq (i-1)c+b$.
If (iii) holds, by Proposition~\ref{cor-Laurent} with $s=i-1$, we can write $Q(d\Mid u^{(i-1)};k^{(i-1)})$ as a Laurent polynomial in $x_{u_{i-1}}$. That is
\begin{align}\label{e-A2-5}
Q(d\Mid u^{(i-1)};k^{(i-1)})&=\frac{E_{u^{(i-1)},k^{(i-1)}}\bigg(P_{\lambda}(x;q,q^c)
P_{\mu}\Big(\Big[\frac{q^{c-b-1}-q^a}{1-q^c}x_0
+\sum_{j=1}^nx_j\Big];q,q^c\Big)\bigg)}
{x_{u_{i-1}}^{|\lambda|+|\mu|-(n-i+1)\big((i-1)c-d\big)}}\\
&\quad \times \prod_{\substack{j=1\\ j\notin U}}^n
\Big(p_j(x_j/x_{u_{i-1}})x_j^{d-(i-1)c}\Big)\times
\prod_{\substack{1\leq u<v\leq n\\u,v\notin U}}\big(x_u/x_v\big)_c
\big(qx_v/x_u\big)_c,\nonumber
\end{align}
where $U:=\{u_1,\dots,u_{i-1}\}$ and the $p_j(z)$ are polynomials in $z$.

By Proposition~\ref{prop-Macdeg} with $s=i-1$ we have that
the degree in $x_{u_{i-1}}$ of $E_{u^{(i-1)},k^{(i-1)}}\big(P_{\lambda}(x;q,q^c)\big)$ is at most $\lambda_1+\cdots+\lambda_{i-1}<|\lambda|$ (because $i\leq \ell(\lambda)$).
Together with the degree in $x_{u_{i-1}}$
of $E_{u^{(i-1)},k^{(i-1)}}\big(P_{\mu}\big(\big[\frac{q^{c-b-1}-q^a}{1-q^c}x_0+\sum_{i=1}^nx_i\big];q,q^c\big)\big)$
is at most $|\mu|$, for $d=(i-1)c$ we have that $Q(d\Mid u^{(i-1)};k^{(i-1)})$ is a Laurent polynomial in $x_{u_{i-1}}$ with degree at most $-\lambda_i-\cdots-\lambda_{\ell(\lambda)}<0$.
It follows that $\CT\limits_x Q\big((i-1)c\Mid u^{(i-1)};k^{(i-1)}\big)=0$.

Now we discuss the cases when $(i-1)c-\lambda_i+1\leq d<(i-1)c$ if $\lambda_i>1$. (For
$\lambda_i=1$ no integer $d$ satisfies $(i-1)c-\lambda_i+1\leq d<(i-1)c$.)
By \eqref{e-Mac4} and \eqref{e-homo-sym} we can write
\begin{multline}\label{e-A2-6}
E_{u^{(i-1)},k^{(i-1)}}\big(P_{\lambda}(x;q,q^c)\big)=P_{\lambda}(x;q,q^c)\Big|_{x_{u_j}=x_{u_{i-1}}q^{k_{i-1}-k_j},1\leq j\leq i-1}\\
=\sum_{\nu}P_{\lambda/\nu}\Big(\Big[\sum_{\substack{j=1\\j\notin U}}^nx_j\Big];q,q^c\Big)
P_{\nu}\big([m_1+\cdots+m_{i-1}];q,q^c\big)x_{u_{i-1}}^{|\nu|},
\end{multline}
where  $\nu$ is over all partitions such that $\nu\subset \lambda$,
and $m_j:=q^{k_{i-1}-k_{j}}$ for $j=1,\dots,i-1$.
We can further obtain that $\ell(\nu)\leq i-1$,
otherwise $P_{\nu}\big([m_1+\cdots+m_{i-1}];q,q^c\big)=0$ by \eqref{e-Mac5}.
By the expression for  $Q(d\Mid u^{(i-1)};k^{(i-1)})$ in \eqref{e-A2-5}, expanding $E_{u^{(i-1)},k^{(i-1)}}\big(P_{\mu}\big(\big[\frac{q^{c-b-1}-q^a}{1-q^c}x_0+\sum_{i=1}^nx_i\big];q,q^c\big)\big)$ as a sum of monomials, using \eqref{e-A2-6} and taking the constant term with respect to $x_{u_{i-1}}$, we can write $\CT\limits_xQ(d\Mid u^{(i-1)};k^{(i-1)})$ as a finite sum of the form
\begin{equation}\label{e-A2-7}
\CT_x\prod_{j\notin U}x_j^{-v_j}P_{\lambda/\nu}\Big(\Big[\sum_{\substack{j=1\\j\notin U}}^nx_j\Big];q,q^c\Big)
\prod_{\substack{1\leq u<v\leq n\\u,v\notin U}}\big(x_u/x_v\big)_c
\big(qx_j/x_i\big)_c,
\end{equation}
where the $v_j$ are integers such that $\sum_{j\notin U}v_j=|\lambda|-|\nu|$ and
$\max\{v_j\}\leq (i-1)c-d\leq \lambda_i-1<\lambda_{\ell(\nu)+1}$ for $\ell(\nu)<i$.
Then for $\lambda_{\ell(\nu)+1}>\max\{v_j\}$ and $\lambda_{\ell(\nu)+1}\geq \lambda_i>1$, we can apply Proposition~\ref{prop-Mac-vanish-2} and find that each constant term of the form
\eqref{e-A2-7} vanishes. Therefore, $\CT\limits_xQ(d\Mid u^{(i-1)};k^{(i-1)})=0$ for $(i-1)c-\lambda_i+1\leq d<(i-1)c$.
\end{proof}

\subsection{Determination of the roots $A_3$ for $A_n(a,b,c,\lambda,\mu)$ if $\ell(\mu)\leq n$}\label{sec-A3-1}

In this subsection, we will determine the roots $A_3$ defined in \eqref{Roots-A} for
$A_n(a,b,c,\lambda,\mu)$ under the assumption $\ell(\mu)\leq n$.
Note that if $\ell(\mu)\leq n$ then all the elements of $A_3$ are negative integers.

\begin{lem}\label{lem-roots-A3-1}
The constant term $A_n(a,b,c,\lambda,\mu)$ vanishes for $a\in A_3$ if $\ell(\mu)\leq n$.
\end{lem}
\begin{proof}
We can write $A_3=\cup_{j=1}^{\ell(\mu)}A_{3,j}$,
where
\[
A_{3,j}=\{-(n-j)c-b-1,-(n-j)c-b-2,\dots,-(n-j)c-b-\mu_{j}\}.
\]
By \eqref{e-relationAQ} we prove this lemma by showing that
\begin{equation}\label{e-A3-1}
\CT_xQ(d)=0 \quad \text{for $-d\in A_{3,j}$ if $1\leq j\leq \ell(\mu)\leq n$}.
\end{equation}

Let $j\in \{1,2,\dots,\ell(\mu)\}$ be a fixed integer.
For $-d\in A_{3,j}$, we have $(n-j)c+b+1\leq d\leq (n-j)c+b+\mu_j$.
Applying the Gessel--Xin operation to $\CT\limits_{x}Q(d)$ gives
\begin{equation}\label{A3-sum-i}
\CT_{x}Q(d)=\sum_{s\in T\subseteq \{1,\dots,n\}}\sum_{\substack{1\leq u_1<\cdots<u_s\leq n\\1\leq k_1,\dots,k_s\leq d}}
\CT_xQ(d\Mid u_1,\dots,u_s;k_1,\dots,k_s).
\end{equation}
As the discussion in the proof of Lemma~\ref{lem-roots-A2i},
it is not hard to show that the $s$ in \eqref{A3-sum-i} can only be $n-j+1$.
Hence, \eqref{A3-sum-i} reduces to
\begin{equation}\label{A3-sum-i-2}
\CT_{x}Q(d)=\sum_{\substack{1\leq u_1<\cdots<u_{n-j+1}\leq n\\1\leq k_1,\dots,k_{n-j+1}\leq d}}\CT_xQ(d\Mid u^{(n-j+1)};k^{(n-j+1)}),
\end{equation}
where $u^{(n-j+1)}=(u_1,\dots,u_{n-j+1})$ and $k^{(n-j+1)}=(k_1,\dots,k_{n-j+1})$.
We prove that every summand $\CT\limits_xQ(d\Mid u^{(n-j+1)};k^{(n-j+1)})$ in \eqref{A3-sum-i-2} vanishes.

Since $(n-j)c+b+1\leq d\leq (n-j)c+b+\mu_j$, we can write $d=(n-j)c+b+t$ for $t\in\{1,\dots,\mu_j\}$.
It follows that $1\leq k_i\leq (n-j)c+b+t$. By Lemma~\ref{cor-key} with $s=n-j+1$, at least one of the following holds:
\begin{enumerate}
\item[(i)] $1\leq k_i\leq b$ for some $i$ with $1\leq i\leq n-j+1$;
\item[(ii)] $-c\leq k_r-k_l\leq c-1$ for some $(r,l)$ such that $1\leq r<l\leq n-j+1$;
\item[(iii)] there exists a permutation $w\in\mathfrak{S}_{n-j+1}$ and nonnegative integers $t_1,\dots,t_{n-j+1}$ such that
\begin{subequations}
\begin{equation*}
k_{w(1)}=b+t_1,
\end{equation*}
and
\begin{equation*}
k_{w(i)}-k_{w(i-1)}=c+t_i \quad \text{for $2\leq i\leq n-j+1$.}
\end{equation*}
\end{subequations}
Here the $t_i$ satisfy
\begin{equation*}
\sum_{i=1}^{n-j+1}t_i\leq t,
\end{equation*}
$w(0):=0$, and $t_i>0$ if $w(i-1)<w(i)$ for $1\leq i\leq n-j+1$.
\end{enumerate}
If (i) holds, then $Q(d\Mid u^{(n-j+1)};k^{(n-j+1)})=0$ since it has a zero factor $(q^{1-k_i})_b$. If (ii) holds, then $Q(d\Mid u^{(n-j+1)};k^{(n-j+1)})=0$ since it has a zero factor
$(q^{k_l-k_r-c})_{2c}$. We omit the details because it is the same as that in the proof of Lemma~\ref{lem-roots-A2i}.
If (iii) holds, we also show that $Q(d\Mid u^{(n-j+1)};k^{(n-j+1)})$ has a zero factor.
By Lemma~\ref{cor-subs} with $s=n-j+1$
\begin{multline*}
-\frac{q^{c-b-1}-q^{-d}}{1-q}x_0-\sum_{i=1}^{n-j+1}\frac{1-q^{c}}{1-q}x_{u_i}
\bigg|_{\substack{d=(n-j)c+b+t, \\[1pt] x_{u_i}=q^{k_{n-j+1}-k_i}x_{u_{n-j+1}},\,0\leq i\leq n-j+1}}\\
=(q^{n_1}+\dots+q^{n_{t-1}})x_{u_{n-j+1}},
\end{multline*}
where $\{n_1,\dots,n_{t-1}\}$ is a set of integers determined by $n,j,b,c$ and the $k_i$.
It follows that
\begin{multline*}
\frac{q^{c-b-1}-q^{-d}}{1-q^c}x_0+\sum_{i=1}^{n}x_i
\bigg|_{\substack{d=(n-j)c+b+t, \\[1pt] x_{u_i}=q^{k_{n-j+1}-k_i}x_{u_{n-j+1}},\,0\leq i\leq n-j+1}}\\
=\frac{1-q}{q^c-1}\big(q^{n_1}+\dots+q^{n_{t-1}}\big)x_{u_{n-j+1}}
+\sum_{i=1,i\notin U}^nx_i,
\end{multline*}
where $U:=\{u_1,u_2,\dots,u_{n-j+1}\}$. Then $Q(d\Mid u^{(n-j+1)};k^{(n-j+1)})$ has the factor
\begin{multline*}
E_{u^{(n-j+1)},k^{(n-j+1)}}\Big(P_{\mu}\big(\big[\frac{q^{c-b-1}-q^{-d}}{1-q^c}x_0+\sum_{i=1}^nx_i\big];q,q^c\big)\Big)\\
=P_{\mu}\Big(\Big[\frac{1-q}{q^c-1}\big(q^{n_1}+\dots+q^{n_{t-1}}\big)x_{u_{n-j+1}}+\sum_{i=1,i\notin U}^nx_i\Big];q,q^c\Big),
\end{multline*}
which is zero for $1\leq t\leq \mu_{j}$ by Proposition~\ref{prop-Mac-vanish} with
$(i,\lambda_i,t)\mapsto (j,\mu_j,q^c)$.
Hence, for $d\in \{(n-j)c+b+1,\dots,(n-j)c+b+\mu_{j}\}$,
we have $\CT\limits_xQ(d\Mid u^{(n-j+1)};k^{(n-j+1)})=0$.
\end{proof}

\subsection{Determination of the roots $A_3$ for $A_n(a,b,c,\lambda,\mu)$ if $\ell(\mu)>n$}\label{sec-A3-2}
In this subsection, we will determine the roots $A_3$ defined in \eqref{Roots-A} for
$A_n(a,b,c,\lambda,\mu)$ under the assumption $\ell(\mu)>n$.

Recall that
\[
A_3=\{-(n-j)c-b-1,-(n-j)c-b-2,\dots,-(n-j)c-b-\mu_{j}\Mid j=1,\dots,\ell(\mu)\}.
\]
Let
\begin{equation}\label{defi-A31}
A_{31}:=\{-(n-j)c-b-1,-(n-j)c-b-2,\dots,-(n-j)c-b-\mu_{j}\Mid j=1,\dots,n\}
\end{equation}
and
\begin{equation}\label{defi-A32}
A_{32}:=\{-(n-j)c-b-1,-(n-j)c-b-2,\dots,-(n-j)c-b-\mu_{j}\Mid j=n+1,\dots,\ell(\mu)\}
\end{equation}
if $\ell(\mu)>n$.
Then $A_3=A_{31}\cup A_{32}$.
Note that all the elements in $A_{31}$(resp. $A_{32}$) are negative (resp. positive) integers if $\ell(\mu)>n$ and $c>b+\lambda_1+\mu_1$.
We can carry out the same discussion as Lemma~\ref{lem-roots-A3-1} to determine the roots
$a\in A_{31}$ for $A_n(a,b,c,\lambda,\mu)$. We state this result in the next lemma and omit the proof.
\begin{lem}\label{lem-roots-A3-21}
If $\ell(\mu)>n$, then
the constant term $A_n(a,b,c,\lambda,\mu)$ vanishes for $a\in A_{31}$.
\end{lem}
Then we prove $A_n(a,b,c,\lambda,\mu)=0$ for $a\in A_{32}$.
That is the content of the next lemma.
\begin{lem}\label{lem-roots-A3-2}
If $\ell(\mu)>n$, then the constant term $A_n(a,b,c,\lambda,\mu)$ vanishes for $a\in A_{32}$.
\end{lem}
\begin{proof}
For $a\in A_{32}$,
we can write $a=(j-n)c-b-t$, where $j\in \{n+1,\dots,\ell(\mu)\}$ and $t\in \{1,\dots,\mu_j\}$ for each $j$.
Note that $a>0$ by the conditions $j\geq n+1$, $t\leq \mu_j$ and $c>b+\lambda_1+\mu_1$.
Then
\begin{align*}
\frac{q^{c-b-1}-q^a}{1-q^c}\Big|_{a=(j-n)c-b-t}
&=\frac{q^{c-b-1}-q^{(j-n)c-b-t}}{1-q^c}\\
&=\frac{q^{c-b-1}-q^{(j-n)c-b-1}+q^{(j-n)c-b-1}-q^{(j-n)c-b-t}}{1-q^c}\\
&=\frac{1-q^{(j-n-1)c}}{1-q^c}q^{c-b-1}+\frac{1-q^{t-1}}{q^c-1}q^{(j-n)c-b-t}.
\end{align*}
Using the above equation, we can write
\begin{multline}\label{e-A3-2-1}
P_{\mu}\Big(\Big[\frac{q^{c-b-1}-q^a}{1-q^c}x_0+\sum_{i=1}^nx_i\Big];q,q^c\Big)\Big|_{a=(j-n)c-b-t}\\
=P_{\mu}\Big(\Big[\frac{1-q^{(j-n-1)c}}{1-q^c}q^{c-b-1}x_0+\frac{1-q^{t-1}}{q^c-1}q^{(j-n)c-b-t}x_0
+\sum_{i=1}^nx_i\Big];q,q^c\Big).
\end{multline}
The right-hand side of \eqref{e-A3-2-1} is in fact of the form
\begin{equation}\label{e2-A3-2-2}
P_{\mu}\Big(\Big[\frac{1-q}{q^c-1}(m_1+\cdots+m_{t-1})
+\sum_{i=1}^nx_i+r_1+\cdots+r_{j-n-1}\Big];q,q^c\Big),
\end{equation}
where $m_i=q^{(j-n)c-b-t+i-1}x_0$ for $i=1,\dots,t-1$,
and $r_i=q^{ic-b-1}x_0$ for $i=1,\dots,j-n-1$.
Then by Proposition~\ref{prop-Mac-vanish},
for $j\in \{n+1,\dots,\ell(\mu)\}$ and $1\leq t\leq \mu_{j}$ the symmetric function \eqref{e2-A3-2-2} vanishes.
It follows that if $\ell(\mu)>n$ and $a\in A_{32}$ we can conclude that
$A_n(a,b,c,\lambda,\mu)=0$
since its factor $P_{\mu}\big(\big[\frac{q^{c-b-1}-q^a}{1-q^c}x_0+\sum_{i=1}^nx_i\big];q,q^c\big)$
vanishes by the above discussion.
\end{proof}

\subsection{The value of $A_n(a,b,c,\lambda,\mu)$ at an additional point}\label{sec-A-addition}

In this subsection, we characterize the expressions for $A_n(a,b,c,\lambda,\mu)$
when $a=-b-1$ if $\ell(\mu)<n$ and $a=(\ell(\mu)-n+1)c-b-1$ if $\ell(\mu)\geq n$ respectively.

We obtain a recursion for $A_n(-b-1,b,c,\lambda,\mu)$ when $\ell(\mu)<n$.
\begin{prop}\label{prop-add-1}
For $\ell(\mu)<n$
\begin{equation}\label{e-add-0}
A_n(-b-1,b,c,\lambda,\mu)=(-1)^bq^{-\binom{b+1}{2}-(b+1)\lambda_n}\frac{1-q^{nc}}{1-q^c}
A_{n-1}(c-b-1,b+c,c,\overline{\mu},\overline{\lambda}),
\end{equation}
where $\overline{\lambda}=(\lambda_1-\lambda_n,\dots,\lambda_{n-1}-\lambda_n)$
and $\overline{\mu}=(\mu_1+\lambda_n,\dots,u_{n-1}+\lambda_n)$.
\end{prop}
\begin{proof}
By Proposition~\ref{prop-equiv}, we have
\begin{equation}\label{e-add-1}
A_n(a,b,c,\lambda,\mu)=\frac{1}{n!}\prod_{i=1}^{n-1}\frac{1-q^{(i+1)c}}{1-q^c}
A'_n(a,b,c,\lambda,\mu),
\end{equation}
where
\begin{multline*}
A'_n(a,b,c,\lambda,\mu)=\CT_x x_0^{-|\lambda|-|\mu|}P_{\lambda}(x;q,q^c)
P_{\mu}\Big(\Big[\frac{q^{c-b-1}-q^a}{1-q^c}x_0+\sum_{i=1}^nx_i\Big];q,q^c\Big)\\
\times \prod_{i=1}^{n}\Big(\frac{x_{0}}{x_{i}}\Big)_a
\Big(\frac{qx_{i}}{x_{0}}\Big)_b\prod_{1\leq i\neq j\leq n}
\Big(\frac{x_{i}}{x_{j}}\Big)_c.
\end{multline*}
We obtain a recursion for $A'_n(-b-1,b,c,\lambda,\mu)$.

We can write
\[
A'_n(-b-1,b,c,\lambda,\mu)=\CT_x N/D,
\]
where
\[
N=x_0^{-|\lambda|-|\mu|}P_{\lambda}(x;q,q^c)
P_{\mu}\Big(\Big[\frac{q^{c-b-1}-q^{-b-1}}{1-q^c}x_0+\sum_{i=1}^nx_i\Big];q,q^c\Big)\\
\times \prod_{i=1}^{n}(qx_i/x_0)_b
\prod_{1\leq i\neq j\leq n}(x_i/x_j)_c
\]
and
\[
D=\prod_{i=1}^{n}(q^{-b-1}x_0/x_i)_{b+1}.
\]
We find that $N$ is a Laurent polynomial in $x_0$ with degree at most $-|\lambda|$, and $D$ is of the form
\eqref{e-formD}. Then, applying Lemma~\ref{lem-prop} to $N/D$ with respect to $x_0$ gives
\begin{multline}\label{e-add-4-1}
A'_n(-b-1,b,c,\lambda,\mu)
=\sum_{u_1=1}^n\sum_{k_1=1}^{b+1}
\CT_x P_{\lambda}(x;q,q^c)
P_{\mu}\Big(\Big[-q^{-b-1}x_{u_1}q^{k_1}+\sum_{i=1}^nx_i\Big];q,q^c\Big)\\
\times \frac{(q^{1-k_1})_b}{(q^{k_1-b-1})_{b+1-k_1}(q)_{k_1-1}(x_{u_1}q^{k_1})^{|\lambda|+|\mu|}}
\times \prod_{\substack{i=1\\i\neq u_1}}^{n}\frac{\big(q^{1-k_1}x_{i}/x_{u_1}\big)_b}
{(q^{k_1-b-1}x_{u_1}/x_i)_{b+1}}
\prod_{1\leq i\neq j\leq n}(x_i/x_j)_c.
\end{multline}
Since the factor $(q^{1-k_1})_b=0$ for $1\leq k_1\leq b$, the summand in the right-hand side of \eqref{e-add-4-1} does not vanish only when $k_1=b+1$.
Therefore, $A'_n(-b-1,b,c,\lambda,\mu)$ reduces to
\begin{multline*}
A'_n(-b-1,b,c,\lambda,\mu)
=C_1
\sum_{u_1=1}^n\CT_x (x_{u_1})^{-|\lambda|-|\mu|}P_{\lambda}(x;q,q^c)
P_{\mu}\Big(\Big[-x_{u_i}+\sum_{i=1}^nx_i\Big];q,q^c\Big)\\
\times \prod_{\substack{i=1\\i\neq u_1}}^{n}\frac{\big(q^{-b}x_{i}/x_{u_1}\big)_b}
{(x_{u_1}/x_i)_{b+1}}
\prod_{1\leq i\neq j\leq n}(x_i/x_j)_c,
\end{multline*}
where $C_1=(-1)^bq^{-\binom{b+1}{2}-(b+1)(|\lambda|+|\mu|)}$.
It is clear that the summand is symmetric for $u_i=1,\dots,n$. Thus, we can further write
$A'_n(-b-1,b,c,\lambda,\mu)$ as
\begin{align}\label{e-add-2-0}
&nC_1\times\CT_x x_{n}^{-|\lambda|-|\mu|}P_{\lambda}(x;q,q^c)
P_{\mu}\Big(\Big[\sum_{i=1}^{n-1}x_i\Big];q,q^c\Big)
\prod_{i=1}^{n-1}\frac{\big(q^{-b}x_{i}/x_n\big)_b}
{(x_n/x_i)_{b+1}}
\prod_{1\leq i\neq j\leq n}
(x_i/x_j)_c\\
&\quad
=nC_1\times\CT_x x_{n}^{-|\lambda|-|\mu|}P_{\lambda}(x;q,q^c)
P_{\mu}\Big(\Big[\sum_{i=1}^{n-1}x_i\Big];q,q^c\Big)\nonumber\\
&\qquad \times
\prod_{i=1}^{n-1}
(q^{b+1}x_n/x_i)_{c-b-1}
\big(q^{-b}x_{i}/x_n\big)_{b+c}
\prod_{1\leq i\neq j\leq n-1}
(x_i/x_j)_c.\nonumber
\end{align}
Here we use the fact that $(x_n/x_i)_{b+1}$ is a factor of $(x_n/x_i)_c$ for $c>b+\lambda_1+\mu_1\geq b$.

Taking $x_nq^{b+1}\mapsto x_0$ in the right-hand side of \eqref{e-add-2-0}, we obtain
\begin{multline*}
A'_n(-b-1,b,c,\lambda,\mu)
=(-1)^bnq^{-\binom{b+1}{2}}
\CT_x
P_{\mu}\Big(\Big[\sum_{i=1}^{n-1}x_i\Big];q,q^c\Big)
P_{\lambda}\Big(\Big[q^{-b-1}x_0+\sum_{i=1}^{n-1}x_i\Big];q,q^c\Big)
\nonumber\\
\times x_{0}^{-|\lambda|-|\mu|}
\prod_{i=1}^{n-1}(x_0/x_i)_{c-b-1}\big(qx_{i}/x_0\big)_{b+c}
\prod_{1\leq i\neq j\leq n-1}
(x_i/x_j)_c.
\end{multline*}
Note that the substitution $x_nq^{b+1}\mapsto x_0$ in the above does not change the constant term.
Using Proposition~\ref{prop-equiv} again, we have
\begin{multline}\label{e-add-3}
A'_n(-b-1,b,c,\lambda,\mu)
=(-1)^bn!q^{-\binom{b+1}{2}}
\prod_{i=1}^{n-2}\frac{1-q^c}{1-q^{(i+1)c}}
\CT_x x_{0}^{-|\lambda|-|\mu|}
P_{\mu}\Big(\Big[\sum_{i=1}^{n-1}x_i\Big];q,q^c\Big)
\\
\times
P_{\lambda}\Big(\Big[q^{-b-1}x_0+\sum_{i=1}^{n-1}x_i\Big];q,q^c\Big)
\prod_{i=1}^{n-1}(x_0/x_i)_{c-b-1}\big(qx_{i}/x_0\big)_{b+c}
\prod_{1\leq i<j\leq n-1}
\Big(\frac{x_{i}}{x_{j}}\Big)_c\Big(\frac{qx_{j}}{x_{i}}\Big)_c.
\end{multline}
From \eqref{e-Mac7},
\begin{multline*}
P_{\mu}\Big(\Big[\sum_{i=1}^{n-1}x_i\Big];q,q^c\Big)P_{\lambda}\Big(\Big[q^{-b-1}x_0+\sum_{i=1}^{n-1}x_i\Big];q,q^c\Big)\\
=(q^{-(b+1)}x_0)^{\lambda_n}P_{\overline{\mu}}\Big(\Big[\sum_{i=1}^{n-1}x_i\Big];q,q^c\Big)
P_{\overline{\lambda}}\Big(\Big[q^{-b-1}x_0+\sum_{i=1}^{n-1}x_i\Big];q,q^c\Big),
\end{multline*}
where $\overline{\lambda}=(\lambda_1-\lambda_n,\dots,\lambda_{n-1}-\lambda_n)$
and $\overline{\mu}=(\mu_1+\lambda_n,\dots,u_{n-1}+\lambda_n)$.
Substituting this into \eqref{e-add-3} gives
\begin{multline}\label{e-add-4}
A'_n(-b-1,b,c,\lambda,\mu)
=(-1)^bn!q^{-\binom{b+1}{2}-(b+1)\lambda_n}
\prod_{i=1}^{n-2}\frac{1-q^c}{1-q^{(i+1)c}}
\CT_x
\frac{P_{\overline{\mu}}\Big(\Big[\sum_{i=1}^{n-1}x_i\Big];q,q^c\Big)}
{x_{0}^{|\lambda|+|\mu|-\lambda_n}}
\\
\times
P_{\overline{\lambda}}\Big(\Big[q^{-b-1}x_0+\sum_{i=1}^{n-1}x_i\Big];q,q^c\Big)
\prod_{i=1}^{n-1}(x_0/x_i)_{c-b-1}\big(qx_{i}/x_0\big)_{b+c}
\prod_{1\leq i<j\leq n-1}
\Big(\frac{x_{i}}{x_{j}}\Big)_c\Big(\frac{qx_{j}}{x_{i}}\Big)_c.
\end{multline}
The constant term in the right-hand side of \eqref{e-add-4} is $A_{n-1}(c-b-1,b+c,c,\overline{\mu},\overline{\lambda})$.
Together with \eqref{e-add-1} (with $a\mapsto -b-1$) gives \eqref{e-add-0}.
\end{proof}

If $\ell(\mu)=n$, then $A_n(-b-1,b,c,\lambda,\mu)$ vanishes by Lemma~\ref{lem-roots-A3-1}.
It also vanishes for $\ell(\mu)>n$ by Lemma~\ref{lem-roots-A3-21}.
Hence, we can not use this value as an additional value to determine
$A_n(a,b,c,\lambda,\mu)$ if $\ell(\mu)\geq n$. For $\ell(\mu)\geq n$, we choose $a=(\ell(\mu)-n+1)c-b-1$ as an additional point to determine $A_n(a,b,c,\lambda,\mu)$.

\begin{prop}\label{prop-add-2}
For $\ell(\mu):=l\geq n$,
\begin{equation}\label{e-add-2}
A_n\big((l-n+1)c-b-1,b,c,\lambda,\mu\big)
=q^tA_n\big(\big(l-n+1\big)c-b-1,b,c,\widetilde{\lambda},\widetilde{\mu}\big),
\end{equation}
where $t=\Big(\binom{l-n+1}{2}c-(b+1)(l-n)\Big)\mu_l$, $\widetilde{\lambda}=(\lambda_1+\mu_{l},\dots,\lambda_n+\mu_{l})$
and $\widetilde{\mu}=(\mu_1-\mu_{l},\dots,\mu_{l-1}-\mu_{l})$.
\end{prop}
\begin{proof}
If $l\geq n$,
then $(l-n+1)c-b-1$ is a nonnegative integer by the assumption $c>b+\lambda_1+\mu_1$.
Since
\[
\frac{q^{c-b-1}-q^{(l-n+1)c-b-1}}{1-q^c}
=q^{c-b-1}\frac{1-q^{(l-n)c}}{1-q^c}=\sum_{i=1}^{l-n}q^{ic-b-1},
\]
we write
\begin{multline}\label{e-add2-1}
A_n\big((l-n+1)c-b-1,b,c,\lambda,\mu\big)
=\CT_x x_0^{-|\lambda|-|\mu|}P_{\lambda}(x;q,q^c)
P_{\mu}\Big(\Big[\sum_{i=1}^{l-n}q^{ic-b-1}x_0+\sum_{i=1}^nx_i\Big];q,q^c\Big)\\
\times
\prod_{i=1}^{n}\Big(\frac{x_{0}}{x_{i}}\Big)_{(l-n+1)c-b-1}
\Big(\frac{qx_{i}}{x_{0}}\Big)_b\prod_{1\leq i<j\leq n}
\Big(\frac{x_{i}}{x_{j}}\Big)_c\Big(\frac{x_{j}}{x_{i}}q\Big)_c.
\end{multline}
By \eqref{e-Mac7}, the right-hand side of \eqref{e-add2-1} can be written as
\begin{multline}\label{e-add2-2}
q^t\CT_x x_0^{-|\lambda|-|\mu|+(l-n)\mu_{l}}P_{\widetilde{\lambda}}(x;q,q^c)
P_{\widetilde{\mu}}\Big(\Big[\sum_{i=1}^{l-n}q^{ic-b-1}x_0+\sum_{i=1}^nx_i\Big];q,q^c\Big)\\
\times
\prod_{i=1}^{n}\Big(\frac{x_{0}}{x_{i}}\Big)_{\big(l-n+1\big)c-b-1}
\Big(\frac{qx_{i}}{x_{0}}\Big)_b\prod_{1\leq i<j\leq n}
\Big(\frac{x_{i}}{x_{j}}\Big)_c\Big(\frac{x_{j}}{x_{i}}q\Big)_c,
\end{multline}
where $t$, $\widetilde{\lambda}$ and $\widetilde{\mu}$ are defined in the proposition.
The constant term in \eqref{e-add2-2} is in fact $A_n\big(\big(l-n+1\big)c-b-1,b,c,\widetilde{\lambda},\widetilde{\mu}\big)$.
\end{proof}

\subsection{A recursion for $A_n(a,b,c,\lambda,\mu)$ and the proof of Theorem~\ref{thm-AFLT}}

Since we have characterized the additional value of $A_n(a,b,c,\lambda,\mu)$ in Proposition~\ref{prop-add-1} and Proposition~\ref{prop-add-2}, and we have determined all the roots of $A_n(a,b,c,\lambda,\mu)$ in Lemmas~\ref{lem-roots-A1}--\ref{lem-roots-A3-2}, we can determine $A_n(a,b,c,\lambda,\mu)$ as a polynomial in $q^a$. In this subsection,
we obtain a recursion for $A_n(a,b,c,\lambda,\mu)$. By the recursion, we give a proof of Theorem~\ref{thm-AFLT}.

We give a recursion for $A_n(a,b,c,\lambda,\mu)$ in the next theorem.
\begin{thm}\label{thm-recur-AFLT}
For $\ell(\mu)\leq n$,
\begin{align}\label{e-rec-A-0}
&A_n(a,b,c,\lambda,\mu)\\
&=(-1)^bq^{-\binom{b+1}{2}-(b+1)(\lambda_n+\mu_n)}\frac{(1-q^{nc})}{(1-q^c)}
\prod_{i=0}^{n-1}\frac{(q^{a+ic+1})_b}{(q^{ic-b})_b}
\prod_{i=1}^{n}\frac{(q^{a+(i-1)c-\lambda_i+1})_{\lambda_i}(q^{ic-b-\lambda_i-\mu_n})_{\mu_n}}{(q^{(i-1)c-b-\lambda_i-\mu_n})_{\lambda_i+\mu_n}}
\nonumber\\
&\quad \times
\prod_{i=1}^n\frac{(q^{a+(n-i)c+b+1})_{\mu_i}}{(q^{(n-i)c})_{\mu_i-\mu_n}(q^{(n-i+1)c+\mu_i-\mu_n})_{\mu_n}}
A_{n-1}(c-b-1,b+c,c,\overline{\mu},\overline{\lambda}),\nonumber
\end{align}
where $\overline{\lambda}=(\lambda_1-\lambda_n,\dots,\lambda_{n-1}-\lambda_n)$
and $\overline{\mu}=(\mu_1+\lambda_n,\dots,u_{n-1}+\lambda_n)$.
\end{thm}
Note that $\ell(\overline{\lambda})\leq n-1$ and $\ell(\overline{\mu})\leq n-1$. Then we can iterate this theorem to obtain a closed form expression for $A_n(a,b,c,\lambda,\mu)$ if $\ell(\mu)\leq n$.
\begin{proof}
Since $A_n(a,b,c,\lambda,\mu)$ is a polynomial in $q^a$ with degree $nc+|\lambda|+|\mu|$ by Corollary~\ref{cor-poly-BC}, we can explicitly determine it if we find all its roots and one additional value.
By Proposition~\ref{prop-add-1} and Lemmas~\ref{lem-roots-A1}--\ref{lem-roots-A3-2}, it is straightforward to
obtain the following recursion for $A_n(a,b,c,\lambda,\mu)$ if $\ell(\mu)<n$:
\begin{align}\label{e-rec-A-1}
A_n(a,b,c,\lambda,\mu)&=\prod_{i=0}^{n-1}\frac{(q^{a+ic+1})_b}{(q^{ic-b})_b}
\prod_{i=1}^{\ell(\lambda)}\frac{(q^{a+(i-1)c-\lambda_i+1})_{\lambda_i}}{(q^{(i-1)c-\lambda_i-b})_{\lambda_i}}
\prod_{i=1}^{\ell(\mu)}\frac{(q^{a+(n-i)c+b+1})_{\mu_i}}{(q^{(n-i)c})_{\mu_i}}\\
&\quad \times(-1)^bq^{-\binom{b+1}{2}-(b+1)\lambda_n}\frac{1-q^{nc}}{1-q^c}
A_{n-1}(c-b-1,b+c,c,\overline{\mu},\overline{\lambda}),\nonumber
\end{align}
where $\overline{\lambda}$ and $\overline{\mu}$ are defined in the theorem.

If $\ell(\mu)=n$ then by Proposition~\ref{prop-add-2} we have
\[
A_n(c-b-1,b,c,\lambda,\mu)=A_n(c-b-1,b,c,\lambda^*,\mu^*),
\]
where $\lambda^*=(\lambda_1+\mu_n,\dots,\lambda_n+\mu_n)$
and $\mu^*=(\mu_1-\mu_n,\dots,\mu_{n-1}-\mu_n)$. Then $\ell(\mu^*)=n-1$. Applying \eqref{e-rec-A-1} to $A_n(c-b-1,b,c,\lambda^*,\mu^*)$ yields
\begin{align*}
A_n(c-b-1,b,c,\lambda,\mu)&=A_n(c-b-1,b,c,\lambda^*,\mu^*)\\
&=\prod_{i=0}^{n-1}\frac{(q^{(i+1)c-b})_b}{(q^{ic-b})_b}
\prod_{i=1}^{n}\frac{(q^{ic-b-\lambda_i-\mu_n})_{\lambda_i+\mu_n}}{(q^{(i-1)c-b-\lambda_i-\mu_n})_{\lambda_i+\mu_n}}
\prod_{i=1}^{n-1}\frac{(q^{(n-i+1)c})_{\mu_i-\mu_n}}{(q^{(n-i)c})_{\mu_i-\mu_n}}
\\
&\times(-1)^bq^{-\binom{b+1}{2}-(b+1)(\lambda_n+\mu_n)}\frac{1-q^{nc}}{1-q^c}A_{n-1}(c-b-1,b+c,c,\overline{\mu},\overline{\lambda}).\nonumber
\end{align*}
Together with the fact that $A_n(a,b,c,\lambda,\mu)$ is a polynomial in $q^a$ with degree $nc+|\lambda|+|\mu|$ and we have determined all its roots by Lemmas~\ref{lem-roots-A1}--\ref{lem-roots-A3-2}, we obtain
\begin{align}\label{e-rec-A-2}
A_n(a,b,c,\lambda,\mu)
&=\prod_{i=0}^{n-1}\frac{(q^{a+ic+1})_b}{(q^{(i+1)c-b})_b}
\prod_{i=1}^{n}\frac{(q^{a+(i-1)c-\lambda_i+1})_{\lambda_i}}{(q^{ic-\lambda_i-b})_{\lambda_i}}
\prod_{i=1}^{n}\frac{(q^{a+(n-i)c+b+1})_{\mu_i}}{(q^{(n-i+1)c})_{\mu_i}}\\
&\quad\times\prod_{i=0}^{n-1}\frac{(q^{(i+1)c-b})_b}{(q^{ic-b})_b}
\prod_{i=1}^{n}\frac{(q^{ic-b-\lambda_i-\mu_n})_{\lambda_i+\mu_n}}{(q^{(i-1)c-b-\lambda_i-\mu_n})_{\lambda_i+\mu_n}}
\prod_{i=1}^{n-1}\frac{(q^{(n-i+1)c})_{\mu_i-\mu_n}}{(q^{(n-i)c})_{\mu_i-\mu_n}}\nonumber\\
&\quad \times(-1)^bq^{-\binom{b+1}{2}-(b+1)(\lambda_n+\mu_n)}\frac{1-q^{nc}}{1-q^c}
A_{n-1}(c-b-1,b+c,c,\overline{\mu},\overline{\lambda})\nonumber
\end{align}
for $\ell(\mu)=n$. Substituting
\begin{equation*}
\frac{(q^{ic-b-\lambda_i-\mu_n})_{\lambda_i+\mu_n}}{(q^{ic-b-\lambda_i})_{\lambda_i}}
=(q^{ic-b-\lambda_i-\mu_n})_{\mu_n}
\end{equation*}
and
\[
\frac{(q^{(n-i+1)c})_{\mu_i-\mu_n}}{(q^{(n-i+1)c})_{\mu_i}}
=\frac{1}{(q^{(n-i+1)c+\mu_i-\mu_n})_{\mu_n}}
\]
into \eqref{e-rec-A-2}, we obtain \eqref{e-rec-A-0}.
\end{proof}

Now we can complete the proof of Theorem~\ref{thm-AFLT}.
\begin{proof}[Proof of Theorem~\ref{thm-AFLT}]
If $\ell(\mu)\leq n$, then $A_n(a,b,c,\lambda,\mu)$ is determined by its recursion \eqref{e-rec-A-0}
and the initial value $A_0(a,b,c,0,0)=1$. Hence, to prove the theorem for $\ell(\mu)\leq n$,
it is straightforward to verify that the right-hand side of \eqref{AFLT} satisfies the same recursion and initial condition. To carry out this verification, one may need \eqref{e-special}.

For $\ell(\mu)=l>n$, we can obtain a formula for $A_n(a,b,c,\lambda,\mu)$ by all its roots (Lemmas~\ref{lem-roots-A1}--\ref{lem-roots-A3-2}) and the value $A_n\big((l-n+1)c-b-1,b,c,\lambda,\mu\big)$.
This additional value $A_n\big((l-n+1)c-b-1,b,c,\lambda,\mu\big)$ is determined by \eqref{e-add-2}
and the recursion for $A_n(a,b,c,\lambda,\mu)$ for $\ell(\mu)\leq n$.

For example, we can get a recursion for $A_n\big(2c-b-1,b,c,\lambda,\mu\big)$ if $\ell(\mu)=n+1$.
By \eqref{e-add-2} we have
\[
A_n\big(2c-b-1,b,c,\lambda,\mu\big)=q^{(c-b-1)\mu_{n+1}}A_n(2c-b-1,b,c,\widetilde{\lambda},\widetilde{\mu}),
\]
where $\widetilde{\lambda}=(\lambda_1+\mu_{n+1},\dots,\lambda_n+\mu_{n+1})$
and $\widetilde{\mu}=(\mu_1-\mu_{n+1},\dots,\mu_{n}-\mu_{n+1})$. Then $\ell(\widetilde{\mu})=n$ and we can apply
\eqref{e-rec-A-0} to $A_n(2c-b-1,b,c,\widetilde{\lambda},\widetilde{\mu})$ to obtain a recursion for
$A_n\big(2c-b-1,b,c,\lambda,\mu\big)$. This leads to a recursion for $A_n\big(a,b,c,\lambda,\mu\big)$ for $\ell(\mu)=n+1$.

By the above argument, to prove the theorem for the $\ell(\mu)>n$ case, it suffices to check that
the right-hand side of \eqref{AFLT} satisfies the same recursion as \eqref{e-rec-A-0} for $\ell(\mu)=n$ and
the same inductive relation as \eqref{e-add-2} for $\ell(\mu)>n$ and $a=\big(\ell(\mu)-n+1\big)c-b-1$.
These verifications are straightforward.
\end{proof}

\subsection*{Acknowledgements}

This work was supported by the
National Natural Science Foundation of China (No. 12171487).

\end{document}